\documentclass[11pt,a4paper]{amsart}

\usepackage[paper=a4paper, text={155mm,208mm},centering]{geometry}

\usepackage{amsfonts, amstext, amsmath, amsthm, amscd, amssymb, amsxtra}
\usepackage{float, placeins}
\usepackage[all]{xy}
\usepackage{enumerate}
\usepackage{mathtools}
\usepackage{graphicx,color}
\usepackage{eucal, xspace}
\usepackage{epic}
\usepackage{eepic}
\usepackage{comment}

\usepackage{tikz}
\usetikzlibrary{cd, arrows}

\theoremstyle{plain}
\newtheorem{theorem}{Theorem}[section]
\newtheorem*{theorem*}{Theorem}
\newtheorem{proposition}[theorem]{Proposition}
\newtheorem*{proposition*}{Proposition}
\newtheorem{corollary}[theorem]{Corollary}
\newtheorem{lemma}[theorem]{Lemma}
\newtheorem*{lemma*}{Lemma}

\newtheorem*{claim*}{Claim}
\theoremstyle{definition}
\newtheorem{definition}[theorem]{Definition}

\newtheorem{remark}[theorem]{Remark}
\newtheorem*{remark*}{Remark}

\def\Z{\mathbb{Z}}
\def\Q{\mathbb{Q}}

\def\C{\mathbb{C}}

\def\Hom{\operatorname{Hom}}

\def\ba{\begin{array}}
\def\ea{\end{array}}

\def\wt{\widetilde}
\def\ol{\overline}
\def\e{\mathrm{e}}

\newcommand{\twotwo}[4]{\left( \begin{array}{cc} #1 & #2 \\ #3 & #4 \end{array} \right)}
\newcommand{\cvec}[2]{\left( \begin{array}{c} #1 \\ #2 \end{array} \right)}
\newcommand{\an}[1]{\langle{#1}\rangle}
\newcommand{\xra}{\xrightarrow}

\DeclareMathOperator{\im}{Im}

\DeclareMathOperator{\Id}{Id}

\DeclareMathOperator{\pr}{pr}

\DeclareMathOperator{\pt}{pt}

\DeclareMathOperator{\Spin}{Spin}
\DeclareMathOperator{\colim}{colim}

\newcommand{\del}{\partial}

\newcommand{\CPI}{\mathbb{CP}^{\infty}}

\newcommand{\SC}{\mathcal{S}^{\rm st}}
\newcommand{\HSC}{\mathcal{S}^{\rm st}_{\rm h}}
\newcommand{\ah}{\alpha^{\mathrm{h}}}

\newcommand{\Spinc}{\Spin^{c}}
\newcommand{\s}{\mathfrak{s}}

\usepackage[colorlinks,
    linkcolor={blue!70!black},
    citecolor={blue!50!black},
    urlcolor={blue!80!black}]{hyperref}

\newcommand{\lmfrac}[2]{\mbox{\small$\displaystyle\frac{#1}{#2}$}}   
\newcommand{\smfrac}[2]{\mbox{\footnotesize$\displaystyle\frac{#1}{#2}$}} 
\newcommand{\tmfrac}[2]{\mbox{\large$\frac{#1}{#2}$}} 

\newcommand{\bsm}{\left(\begin{smallmatrix}}
\newcommand{\esm}{\end{smallmatrix}\right)}

\providecommand{\coker}{\mathop{\rm coker}\nolimits}
\providecommand{\Hom}{\mathop{\rm Hom}\nolimits}

\begin{document}

\title[Simply-connected manifolds with large homotopy stable classes]
{Simply-connected manifolds with large homotopy stable classes}

\author[A.~Conway]{Anthony Conway}
\address{Massachusetts Institute of Technology, Cambridge MA 02139}
\email{anthonyyconway@gmail.com}
\author[D.~Crowley]{Diarmuid Crowley}
\address{School of Mathematics and Statistics, University of Melbourne,
Australia}
\email{dcrowley@unimelb.edu.au}
\author[M.~Powell]{Mark Powell}
\address{Department of Mathematical Sciences, Durham University, United Kingdom}
\email{mark.a.powell@durham.ac.uk}
\author[J.~Sixt]{Joerg Sixt}
\email{sixtj@yahoo.de}

\def\subjclassname{\textup{2020} Mathematics Subject Classification}
\expandafter\let\csname subjclassname@1991\endcsname=\subjclassname
\expandafter\let\csname subjclassname@2000\endcsname=\subjclassname

\begin{abstract}
For every $k \geq 2$ and $n \geq 2$ we construct $n$ pairwise homotopically inequivalent simply-connected,
closed $4k$-dimensional manifolds, all of which are stably diffeomorphic to one another.
Each of these manifolds has hyperbolic intersection form and is stably parallelisable.
In dimension~$4$, we exhibit an analogous phenomenon for spin$^{c}$\hspace{-1pt}  structures on~$S^2 \times S^2$.

For $m\geq 1$, we also provide similar $(4m{-}1)$-connected $8m$-dimensional examples,
where the number of homotopy types in a stable diffeomorphism class
is related to the order of the image of the stable $J$-homomorphism $\pi_{4m-1}(SO) \to \pi^s_{4m-1}$.
\end{abstract}

\maketitle

\section{Introduction}
Let $q$ be a positive integer and let
$W_g := \#_g (S^{q} \times S^{q})$ be the $g$-fold connected sum of the manifold $S^{q} \times S^{q}$ with itself.
Two compact, connected smooth $2q$-manifolds~$M_0$ and $M_1$ with the same Euler characteristic are
\emph{stably diffeomorphic}, written $M_0 \cong_{\text{st}} M_1$, if there exists a non-negative integer~$g$ and a diffeomorphism
\[ M_0 \# W_g \to M_1 \# W_g. \]
Note that $S^{q} \times S^{q}$ admits an orientation-reversing diffeomorphism.
Hence the same is true of $W_g$ and it follows that when the $M_i$ are orientable the diffeomorphism type of the
connected sum does not depend on orientations.

A paradigm of modified surgery, as developed by Kreck~\cite{KreckSurgeryAndDuality}, is that one first seeks to
classify $2q$-manifolds up to stable diffeomorphism, and then for each $M_0$ one tries to
understand its \emph{stable class}:
\[ \SC(M_0) := \{M_1  \mid M_1 \cong_{\text{st}} M_0 \}/\text{diffeomorphism.}\]
The efficacy of this method was first demonstrated by Hambleton and Kreck, who applied it to $4$-manifolds with finite fundamental group in a series of papers~\cite{HambletonKreckFiniteGroup, HambletonKreckSmoothStructures, HambletonKreckCancellation2, HambletonKreckCancellation3}.

On the other hand, the Browder-Novikov-Sullivan-Wall surgery exact sequence~\cite{WallSurgery} aims instead to classify manifolds within a fixed homotopy class.  In general there is no obvious relationship between homotopy equivalence and stable diffeomorphism, although in some 
cases there are implications e.g.~\cite{DavisBorelNovikov}.
To enable a comparison between the two approaches, we define the {\em homotopy stable class} of $M_0$ to be
\[ \SC_{\rm h}(M_0)=\lbrace M_1 \mid M_1
\cong_{\text{st}} M_0 \rbrace / \text{homotopy equivalence}. \]
Our aim is to investigate the cardinality of  $\SC_{\rm h}(M_0)$, and in particular we shall exhibit
new examples of simply-connected manifolds with arbitrarily large homotopy stable class.

\medskip

Throughout this article we will consider closed, connected, 
simply-connected, smooth manifolds.
In order to define the intersection form and related invariants we orient all manifolds.
When necessary, to achieve unoriented results, we will later factor out by the effect of the choice of orientation.

When the dimension is $4k{+}2$, Kreck showed that the stable class of such manifolds is trivial~\cite[Theorem D]{KreckSurgeryAndDuality}.
We therefore focus on dimensions $4k$ with $k>1$ (dimension 4 will be discussed separately below).
 Kreck also showed that for every such simply-connected manifold $M^{4k}$, the stable class of $M^{4k} \# W_1$ is trivial. But as pointed out by Kreck and Schafer~\cite[I]{KreckSchafer}, for $k \! > \! 1$ examples of closed, simply-connected $(2k{-}1)$-connected $4k$-manifolds $M$ with arbitrarily large homotopy stable class
have been implicit in the literature since Wall's classification of these manifolds up to
the action of the group of homotopy spheres \cite{Wall-classification-n-1-conn-2n}.
These examples are distinguished by their intersection form
\[ \lambda_M \colon H_{2k}(M; \Z) \times H_{2k}(M; \Z) \to \Z,\]
which must be definite (in order to have inequivalent forms) and in order to realise the forms by closed, almost-parallelisable manifolds they must have signature divisible by $8 |bP_{4k}|$, where $|bP_{4k}|$ is the order of the group of homotopy $(4k{-}1)$-spheres which bound parallelisable manifolds~\cite[Corollary~on~p.~457]{Kervaire-Milnor-ICM}.

In this paper we consider examples where the intersection form is isomorphic to the standard hyperbolic form
\[ H^+(\Z) = \left(\Z^2, \twotwo{0}{1}{1}{0} \right)   \]
and where there is an additional invariant, a homomorphism $f \colon \Z^2 \to \Z$.
The pair $(H^+(\Z),f)$ is an example of an \emph{extended symmetric form}; see Definition~~\ref{defn:extended-symm-form}.
The isometries of the rank two hyperbolic form are highly restricted: they are generated by switching the two basis vectors and multiplying both basis vectors by $-1$.  As such the unordered pair
\[\{a,b\} := \{f(1,0),f(0,1)\}/(\pm 1),\]
considered up to  multiplication of both integers by $-1$, gives an invariant of the isometry class of the extended symmetric form $(H^+(\Z),f)$.
On the other hand, in the Witt class, or stable equivalence class,
only the divisibility $d:= \gcd(a,b)$ and the product $A=ab$ are invariants.
Since a fixed number $A$ can often be factorised in many ways as a product of coprime integers $a,b$, if we can define a suitable $f$, this simple algebra has the chance to detect large stable classes.
In the proof of our first main theorem, we will define such an $f$ using 
the cohomology ring of the manifolds we construct.

\begin{theorem} \label{thm:LargeStableClassIntro}
Fix positive integers $n$ and $k \geq 2$.
There are infinitely many stable diffeomorphism classes of closed, smooth, simply-connected $4k$-manifolds
$\{[M_i]_{\rm st} \}_{i=1}^\infty$,
such that $|\SC_h(M_i)| \geq n$.
Moreover $\SC_h(M_i)$ contains a
subset $\{M^j_i\}_{j=1}^n$ of cardinality $n$, where $M^1_i = M_i$, and each $M^j_i$ is stably parallelisable and has hyperbolic intersection form.
%
\end{theorem}

Here \emph{stably parallelisable} means that the tangent bundle becomes trivial after taking the Whitney sum with a trivial bundle of sufficiently high rank.  More than one notion of stabilisation appears in this article, one for manifolds and one for vector bundles.

Kreck and Schafer~\cite{KreckSchafer} constructed examples of $4k$-manifolds $M$ with nontrivial finite fundamental groups, such that the homotopy stable class of $M$ contains distinct elements with hyperbolic intersection forms.  However as far as we know our construction gives the first simply-connected examples and the first for which the homotopy stable class has been shown to have arbitrary cardinality.
In a companion paper \cite{CCPS2}, we will investigate the homotopy stable class in more detail, also for manifolds with nontrivial fundamental group, and we shall relate the homotopy stable class to computations of the $\ell$-monoid from~\cite{CrowleySixt}.

The manifolds we construct in order to prove Theorem~~\ref{thm:LargeStableClassIntro} are shown to be homotopically inequivalent using their cohomology rings.  An alternative construction to obtain nontrivial homotopy stable class instead uses Pontryagin classes to define the homomorphism $f$ in an extended symmetric form. This was alluded to in~\cite{KreckSchafer}, but not carried through.
Section~\ref{s:3c8m-thm} proves a theorem which implies the following result.

\begin{theorem}\label{thm:3-conn-8-manifolds-examples}
For every $m \geq 1$ there exists a pair of closed, smooth, $(4m{-}1)$-connected $8m$-manifolds $M_1$ and $M_2$ with hyperbolic intersection forms, that are stably diffeomorphic but not homotopy equivalent.
\end{theorem}

Compared with the manifolds from Theorem~\ref{thm:LargeStableClassIntro} (for even $m$, in the notation of that theorem), the manifolds $M_1$ and $M_2$ from Theorem~\ref{thm:3-conn-8-manifolds-examples} are not stably parallelisable, but on the other hand since they are $(4m{-}1)$-connected and have the same intersection pairing, their cohomology rings are isomorphic. In particular, once again the intersection form does not help.

To show that the manifolds in Theorem~\ref{thm:3-conn-8-manifolds-examples} are not homotopy equivalent,
we use Wall's homotopy classification of $(4m{-}1)$-connected $8m$-manifolds \cite[Lemma 8]{Wall-classification-n-1-conn-2n}, which makes use of an extended symmetric form $(H^+(\Z),f \colon \Z^2 \to \Z/j_m)$,
where $j_m$ is the order of the image of the stable
$J$-homomorphism $J \colon \pi_{4m-1}(SO) \to \pi^s_{4m-1}$; see Section~\ref{s:3c8m-thm}.

\begin{remark}
The limiting factor preventing us from exhibiting arbitrarily large homotopy stable classes in Theorem~\ref{thm:3-conn-8-manifolds-examples} is that our lower bound depends only on the number of primes dividing $j_m$. This grows with $m$, but in a fixed dimension cannot be made arbitrarily large. On the other hand, if we instead count diffeomorphism classes, then we show in  Theorem~\ref{thm:4k-1-conn-8k-manifolds-examples}~\eqref{item-thm-4k-1-conn-8k-i} that the stable class can be arbitrarily large for $(4m{-}1)$-connected $8m$-manifolds with hyperbolic intersection forms.
\end{remark}

\subsection*{Dimension 4}
Dimension 4 was absent from  the above discussion.
This is because closed, smooth, 
simply-connected 4-manifolds $M$ and $N$ are stably diffeomorphic if and only if they are homotopy equivalent. Here is an outline of why this holds. First, two such 4-manifolds are stably diffeomorphic if and only if there are orientations such that they have the same signatures, Euler characteristics, and $w_2$-types i.e.\ $\sigma(M)=\sigma(N)$, $\chi(M)=\chi(N)$, and their intersection forms have the same parity (even or odd). Thus homotopy equivalence implies stable diffeomorphism. For the other direction, $\sigma(M)=\sigma(N)$ and $\chi(M)=\chi(N)$ implies that the intersection forms are either both definite or both indefinite. In the definite case, the intersection forms must be diagonal by Donaldson's theorem~\cite{Donaldson}, and so the intersections forms are isometric and therefore the manifolds are homotopy equivalent~\cite{Whitehead-4-complexes,Milnor-simply-connected-4-manifolds}. In the indefinite case, the intersection form is determined up to isometry by its rank, parity, and signature, and so again~$M$ and~$N$ are homotopy equivalent. Thus the assumption that $k \geq 2$ was essential in Theorem~\ref{thm:LargeStableClassIntro}.

One way in which an analogous phenomenon does occur in dimension 4 is by considering spin$^{c}$ structures. Seiberg-Witten invariants of 4-manifolds and Heegaard-Floer cobordism maps are indexed by spin$^{c}$ structures.
The first Chern class $c_1$ of the spin$^{c}$ structure then defines the map $f$ in the extended symmetric forms.
We illustrate this in Section~\ref{section:spin-c-structures}, using the 4-manifold $S^2 \times S^2$.

\begin{theorem}\label{thm:spin-c-s2xs2-intro}
   Let $C \in \Z$ with $|C| \geq 16$ and $8 \mid C$. Define $P(C)$ to be the number of distinct primes dividing $C/8$.
  There are $n := 2^{P(C)-1}$ stably equivalent spin$^{c}$\hspace{-1pt} structures $\s_1,\dots,\s_n$ on $S^2 \times S^2$  with $c_1(\s_i)^2 = C \in H^4(S^2 \times S^2) \cong \Z$, that are all pairwise inequivalent.
\end{theorem}

\subsection*{Organisation}

Section~\ref{section:proof-4k-manifolds-thm} proves Theorem~\ref{thm:LargeStableClassIntro},
Section~\ref{s:3c8m-thm} proves Theorem~\ref{thm:3-conn-8-manifolds-examples}, and
Section~\ref{section:spin-c-structures} proves Theorem~\ref{thm:spin-c-s2xs2-intro}.

\subsection*{Conventions}
Throughout this paper all manifolds are compact, simply-connected, and smooth.
As mentioned above we will also equip our manifolds with an orientation.
For the remainder of this paper all (co)homology groups have integral coefficients.  We write $\mathbb{N}_0 := \mathbb{N} \cup \{0\}$.

\subsection*{Acknowledgements}
We would like to thank Manuel Krannich for advice about the homotopy sphere $\Sigma_Q$,
Jens Reinhold for comments on an earlier draft of this paper, and Csaba Nagy for pointing out
a mistake in a previous version of the proof of Theorem~\ref{thm:stable-classification}.

MP is grateful to the Max Planck Institute for Mathematics in Bonn, where he was a visitor while this paper was written.
MP was partially supported by EPSRC New Investigator grant EP/T028335/1 and EPSRC New Horizons grant EP/V04821X/1.

\section{Simply-connected~$4k$-manifolds with arbitrarily large stable class}\label{section:proof-4k-manifolds-thm}

We prove Theorem~\ref{thm:LargeStableClassIntro} by stating and proving Proposition~\ref{prop:ManifoldNab} below.  In the proposition, we construct a collection of $4k$-manifolds $N_{a,b}$, for each unordered pair of positive integers $\{a, b\}$ such that $(2k)!$ divides~$2ab$.   If $\{a,b\} \neq \{a',b'\}$, 
then $N_{a,b}$ and $N_{a',b'}$ are not homotopy equivalent. On the other hand $ab = a'b'$ if and only if $N_{a,b}$ and $N_{a',b'}$ are stably diffeomorphic.  Moreover every manifold $N_{a,b}$ is closed, simply-connected, has hyperbolic intersection form, and is stably parallelisable.  Thus the proposition immediately implies Theorem~\ref{thm:LargeStableClassIntro}.

First we have a lemma. In order to rule out orientation-reversing homotopy equivalences, we shall appeal to the following observation.

\begin{lemma}
\label{lem:OriReversal}
Let~$N$ and~$N'$ be closed, oriented~$4k$-manifolds.
Suppose that a class~$z$ freely generates~$H^2(N)$ and satisfies that $z^{2k}=n$ for some
nonzero~$n\in \Z = H^{4k}(N)$, and similarly for $(N',z')$.
Then any homotopy equivalence~$f \colon N \to N'$ must be orientation preserving.
\end{lemma}

\begin{proof}
Assume that~$f$ is of degree~$\varepsilon=\pm 1$.
Since~$f$ is a homotopy equivalence,~$N$ and~$N'$ have isomorphic cohomology rings.
In particular~$H^2(N')\cong \Z$ is generated by~$z'= (f^{*})^{-1}(z)$.
Since~${z'}^{2k}=n$ in~$H^{4k}(N') \cong \Z$, and~$f^*({z'}^{2k})=z^{2k}$, properties of the cap and cup products show that
$$n=f_*(z^{2k}\cap [N])=f_*(f^*({z'}^{2k}) \cap [N])={z'}^{2k} \cap f_*([N])={z'}^{2k} \cap \varepsilon [N']=\varepsilon n.$$
Since~$n\neq 0$, this implies that~$f$ must be orientation-preserving.
%
\end{proof}

Now we proceed with the construction of the promised manifolds.

\begin{proposition} \label{prop:ManifoldNab}
Fix $k >1$. Given an unordered pair~$\{a, b\}$ of positive coprime integers such that~$(2k)!$ divides~$2ab$, there exists a closed, oriented, $4k$-manifold~$N^{4k}_{a,b}$ with the following properties.
\begin{enumerate}[(i)]
\item The manifold~$N_{a, b}$ is simply-connected and stably parallelisable.
\item The ring~$H^*(N_{a,b})$ has generators~$w, x, y, z$ and~$1$ of degrees~$2k{+}2$,~$2k$,
$2k$,~$2$ and~$0$ respectively,
with~$z^k = ax + by, x^2 = 0 = y^2,2abw=z^{k+1}, xz = bw, yz = aw$ and~$xy$
generates~$H^{4k}(N_{a, b})$.
\end{enumerate}
In particular, the intersection form of~$N_{a, b}$ is hyperbolic and~$z^{2k} = 2ab xy$ is~$2ab$
times a fundamental class of~$N_{a, b}$.
If~$\{a, b\} \neq \{a', b'\}$ then~$N_{a, b}$ and~$N_{a', b'}$ have non-isomorphic integral cohomology rings
and so are not homotopy equivalent. Moreover $ab = a'b'$ if and only if
$N_{a,b}$ and~$N_{a',b'}$ are stably diffeomorphic.
\end{proposition}

\begin{proof}
Note that if we have a manifold $N_{a,b}$ and if we choose a stable normal framing on~$N_{a, b}$, then the pair~$(N_{a, b}, z)$ corresponds to a (normally) framed manifold over~$\C P^{\infty}$ using the identification~$H^2(N_{a,b}) \cong [N_{a,b},\C P^{\infty}]$.   This motivates the method we shall use, constructing~$N_{a, b}$ by framed surgery on stably normally framed manifolds over~$\C P^{\infty}$. It will then follow automatically that the manifolds we obtain are stably parallelisable, since a manifold with trivial stable normal bundle has trivial stable tangent bundle too.

We start with~$S^2$ together with the unique framing of its stable normal bundle corresponding to a choice of orientation, and consider the corresponding dual orientation class~$\alpha \in H^2(S^2)$.
Take the~$2k$-fold product of~$S^2$ with itself, \[X_{0} : = S^2 \times \cdots \times S^2,\]
and define~$\beta_0 \in H^2(X_{0})$
to be the class that restricts to~$\alpha$ in each~$S^2$ factor. This means that under the inclusion
\[\iota_{j} \colon \{*\} \times \cdots \times S^2 \times \cdots \{*\} \to S^2 \times \cdots \times S^2\] in the ~$j$th factor,~$\iota_j^*(\beta_0) = \alpha$. Equivalently, let~$p_i \colon S^2 \times \cdots \times S^2 \to S^2$ be the~$i$th projection. Then~$\beta_0 = \sum_{i=1}^{2k} p_i^*(\alpha)$.
An elementary calculation shows that
\[ \beta_0^{2k} = (2k)![X_0]^* \in H^{4k}(X_0).\]
Here we write~$[X_0]^* \in H^{4k}(X_0)$ for the dual of the
fundamental class~$[X_0] \in H_{4k}(X_0)$.
To make this calculation, use~$\beta_0 = \sum_{i=1}^{2k} p_i^*(\alpha)$ and note that:
\begin{enumerate}[(i)]
  \item~$p_i^*(\alpha) \cup p_j^*(\alpha) = p_j^*(\alpha) \cup p_i^*(\alpha)$ for~$i \neq j$,
  \item~$p_i^*(\alpha) \cup p_i^*(\alpha) = p_i^*(\alpha \cup \alpha) = p_i^*(0) =0$, and
  \item ~$p_1^*(\alpha) \cup \cdots \cup p_{2k}^*(\alpha) = [X_0]^*$.
\end{enumerate}
By assumption there is a positive integer~$j$ such that~$2ab = j(2k)!$.  Take
$X_1 := \#^j X_0$ to be the framed~$j$-fold connected sum of~$X_0$ and
$\beta_1 \in H^2(X_1)$ to be the class that restricts to~$\beta_0$ in each summand.  That is,~$H^2(X_1) \cong \bigoplus^{j} H^2(X_0)$ and~$\beta_1 = (\beta_0,\dots,\beta_0)$.
Then
$$ \beta_1^{2k} = j\beta_0^{2k} = j(2k)! [X_1]^* =  2ab[X_1]^* \in H^{4k}(X_1).$$

The element $\beta_1 \in H^2(X_1)$ and the normal framing on~$X_1$ defines a normal map~
\[(\beta_1, \ol{\beta}_1) \colon X_1 \to \C P^{\infty},\]
where we
take the trivial bundle over~$\C P^{\infty}$.
By surgery below the middle dimension, the normal map~$(\beta_1, \ol{\beta}_1)$ is normally bordant
to a~$2k$-connected map~$(\beta_2, \ol{\beta}_2) \colon X_2 \to \C P^{\infty}$.
Since~$X_0$ has signature zero, the same holds for~$X_1$ and~$X_2$.
Since the stable normal bundle of~$X_2$ is framed, so is the stable tangent bundle. Therefore the stable tangent bundle has trivial~$2k$-th Wu class vanishes and so the intersection form on~$X_2$ is even.
Let~$z_\infty \in H^2(\C P^\infty)$ be the generator restricting to~$\alpha \in H^2(\C P^1) = H^2(S^2)$ via the inclusion~$\C P^1 \to \C P^{\infty}$,
and consider the Poincar\'e dual of~$\beta_2^*(z_\infty^k)$,
$$ u := \mathrm{PD}(\beta_2^*(z_\infty^k)) \in H_{2k}(X_2).$$
Since~$\beta_2 \colon X_2 \to \C P^{\infty}$ is
$2k$-connected,~$H_{2k}(X_2) \to H_{2k}(\C P^{\infty}) \cong \Z$
is onto and therefore splits since $\Z$ is free.
Since all homology groups are torsion-free, the dual map can be identified with the map $\beta_2^* \colon H^{2k}(\C P^{\infty}) \to H^{2k}(X_2)$ on cohomology. The splitting for $\beta_2$ dualises to a splitting for $\beta_2^*$, so the image of a generator $\beta_2^*(z^k_{\infty})$ generates a summand.  Applying Poincar\'e duality we see that~$u \in H_{2k}(X_2)$ is a primitive element; i.e.\ ~$u$ generates a summand of~$H_{2k}(X_2)$.

We take connected sum with an additional copy of~$S^{2k} \times S^{2k}$ with null-bordant framing
and trivial map to~$\C P^{\infty}$ to obtain \[X_3 := X_2 \# (S^{2k} \times S^{2k})\] and
a normal map~$(\beta_3, \ol{\beta}_3) \colon X_3 \to \C P^{\infty}$.
Note that up until this point we have only used the product $ab$, rather than the data of the pair $\{a,b\}$. This will change for the upcoming construction of $X_4 = N_{a,b}$.

The intersection form $\lambda_{X_3}$ of~$X_3$ has an orthogonal decomposition corresponding to
the connected sum decomposition of~$X_3$:
$$ (H_{2k}(X_3), \lambda_{X_3}) = (H_{2k}(X_2), \lambda_{X_2}) \oplus H^+(\Z),$$
where~$H^+(\Z)$ is the standard symmetric hyperbolic form.
Let~$\{e, f\}$ be a standard basis for~$H^+(\Z)$.
Since $a$ and $b$ are coprime, we may and shall choose integers~$c, d$ such that~$ad-bc = 1$.
We also write~$u = \mathrm{PD}(\beta_3^*(z_\infty^k))$.
Here note that $u$ is essentially the same element as the element $u \in H_2(X_2)$ that we defined above thinking of $H_2(X_2)$ as a subgroup of $H_2(X_3)$.
Keeping this in mind, we have that
\[
\lambda_{X_3}(u, u) = \an{ \beta_3^*(z^k_{\infty}) \cup \beta_3^*(z^k_{\infty}), [X_3]} = \an{ \beta_3^*(z_\infty^{2k}), [X_3] } = \an{z^{2k}_{\infty},(\beta_3)_*[X_3] }  =  2ab,
\]
since~$z_{\infty}^{2k}$ generates~$H^{4k}(\C P^{\infty})$ and since~$(\beta_3)_*$ sends~$[X_3]$ to the same multiple of the generator of~$H_{4k}(\C P^{\infty})$ as~$(\beta_1)_*$ sends~$[X_1]$ to.
Since~$u \in H_{2k}(X_2) \subseteq H_{2k}(X_3)$ is primitive and since $\lambda_{X_2}$ is nonsingular, there is an element~$v'' \in H_{2k}(X_2) \subseteq H_{2k}(X_3)$ such that
\[
 \lambda_{X_3}(u, v'') =\lambda_{X_2}(u,v'')= 1.
 \]
Now set~$v' := (ad + bc)v''$ as well as
\[v : = v' + e + \frac{2 cd - \lambda_{X_3}(v', v')}{2} f.\]
Since $u \in H_2(X_2)$ and $e,f \in H^+(\Z)$, we observe that~$\lambda_{X_3}(u, e) = \lambda_{X_3}(u, f) = 0$.
As a consequence, the elements~$u, v$ span a subspace~$H_{u, v} \subseteq H_{2k}(X_3)$ where~$\lambda_{X_3}$
restricted to~$H_{u, v}$ has matrix
$$A= \twotwo{2ab}{ad + bc}{ad+bc}{2cd},$$
which has determinant~$4abcd - (ad+bc)^2 = -(ad-bc)^2 = -1$.
Hence~$H_{u, v}$ is an orthogonal summand of~$(H_{2k}(X_3), \lambda_{X_3})$ and a calculation shows that~$H_{u, v}$ is hyperbolic with standard basis~$\{e_1, f_1\}$ where
$u = ae_1 + bf_1$ and $v=ce_1 +d f_1$. To see this, let~$P:= \bsm a&b \\ c&d \esm$ and note that~$P \bsm 0&1 \\ 1&0 \esm P^T =A$.

The orthogonal complement of~$H_{u, v}$, namely~$H_{u, v}^\perp$, has signature equal to the signature
of~$X_3$, which is zero and hence since the intersection form is even,~$H_{u, v}^\perp$ is stably hyperbolic.

We assert that~$H_{u, v}^\perp$ maps trivially to~$H_{2k}(\C P^{\infty})$ under $\beta_{3*}$.
To see this, first note that~$H^{2k}(\C P^{\infty}) \cong \Z$, generated by~$z^k_{\infty}$.
We have an isomorphism
\[z^k_{\infty} \cap - \colon H_{2k}(\C P^{\infty}) \xrightarrow{\cong} H_0(\CPI) \cong \Z.\]
Recall that now~$u = \mathrm{PD}(\beta_3^*(z_\infty^k)) \in H_{u,v}$ and let~$x \in H_{u, v}^\perp$. Then
\[0 = \lambda_{X_3}(u,x) = \mathrm{PD}^{-1}(u) \cap x = \beta_3^*(z_\infty^k) \cap x = z_{\infty}^k \cap (\beta_{3})_*(x). \]
Since~$z^k_{\infty} \cap -$ is an isomorphism, this implies that~$(\beta_{3})_*(x)=0$, which proves the assertion.

Now, since~$\beta_3 \colon X_3 \to \C P^{\infty}$ is ~$2k$-connected
and since~$H_{u, v}^\perp$ maps trivially to~$H_{2k}(\C P^{\infty})$, the Hurewicz theorem
and the linked long exact sequences
$$
\xymatrix{
\cdots \ar[r] &
\pi_{2k+1}(\C P^\infty, X_3) \ar[d]^\cong \ar[r] &
\pi_{2k}(X_3) \ar[d] \ar[r] &
\pi_{2k}(\C P^\infty) \ar[d] \ar[r] &
\cdots \\
\cdots \ar[r] &
H_{2k+1}(\C P^\infty, X_3) \ar[r] &
H_{2k}(X_3) \ar[r]^-{(\beta_3)_*} &
H_{2k}(\C P^\infty) \ar[r] &
\cdots
}
$$
show that every element of~$H_{u, v}^\perp$ is represented by a map from a~$2k$-sphere in~$\pi_{2k}(X_3)$.
Hence standard surgery arguments allow us to perform framed surgery on
$(\beta_3, \ol{\beta}_3) \colon X_3 \to \C P^{\infty}$ to kill~$H_{u, v}^\perp$.
We obtain a normal map~$(\beta_4, \ol{\beta}_4) \colon X_4 \to \C P^{\infty}$,
with intersection form isomorphic to~$(H_{u, v}, \lambda_{X_3}|_{H_{u, v}})$.
The manifold \[N_{a,b} : = X_4\]
 is the required manifold, as we verify next.  For the rest of the proof we shall write~$N := N_{a,b}$ for brevity.
 We use the orientation corresponding to the fundamental class $[N_{a,b}]$ induced from tracking $[X_0]$ through the construction.

We have already noted at the beginning of the proof that the construction via normally framed surgery implies that $N_{a,b}$ is stably parallelisable.
As the map~$\beta_4 \colon N_{a,b} \to \C P^\infty$ is~$2k$-connected and since there is an isomorphism
$\theta \colon H_{2k}(N_{a,b}) \to H_{u, v} \cong \Z^2$, the manifold~$N_{a,b}$ is simply-connected and
has the correct integral (co)homology groups.
To verify that~$N_{a,b}$ has the required cohomology ring we set
\[z := \beta_4^*(z_\infty),
\quad x := \mathrm{PD}^{-1}(\theta^{-1}(e_1)),
\quad y := \mathrm{PD}^{-1}(\theta^{-1}(f_1)).\]
Since~$u = ae_1 + b f_1$, it follows that~$z^k = a x + by$.
Since~$\theta^{-1}(e_1), \theta^{-1}( f_1)$ form
a standard hyperbolic basis for~$(H_{2k}(N_{a,b}), \lambda_{N_{a,b}})$,
it follows that~$xy$ generates~$H^{4k}(N_{a,b})$ and $z^{2k} \cap [N_{a,b}] >0$.
Finally, since~$z^{k-1}$ generates~$H^{2k-2}(N_{a,b}) \cong \Z$,
there is a generator~$w \in H^{2k+2}(N_{a,b})$ such that~$z^{k-1}w = xy$.
The remaining properties of~$H^*(N_{a,b})$ follow from Poincar\'e duality.

Finally, let $\an{z^k} \subseteq H^{2k}(N_{a,b})$ be the subgroup generated by~$z^k$
and consider the isomorphism class of the pair~$(H^{2k}(N_{a,b}), \an{z^k})$.
This pair, modulo the action of the self-equivalences of~$N_{a,b}$ on~$H^{2k}(N_{a,b})$,
is a homotopy invariant of~$N_{a,b}$.
Since~$z^{2k} \neq 0$, and since $z^{2k} \cap [N_{a,b}] >0$, every self-homotopy equivalence of $N_{a,b}$
is orientation preserving by Lemma~\ref{lem:OriReversal}.

Thus $\an{z^k}$ modulo the action of~$\mathrm{Aut}(H^+(\Z))$
is a homotopy invariant.  We claim that the pair~$\{a, b\}$ is an invariant of this action.
To see this, from the form of the matrix~$A$ above, it is easy to see that the automorphisms of the hyperbolic form are  \[\pm \Id \text{ and } \pm \twotwo{0}{1}{1}{0}.\]
So automorphisms can change the sign of both $a$ and $b$ simultaneously, and they can switch $a$ and $b$. Then since we always take~$a,b>0$, the unordered pair of positive integers~$\{a,b\}$ is an invariant of the homotopy type. Hence if there is a homotopy equivalence~$N_{a, b} \to N_{a', b'}$, then
we have~$\{a, b\} = \{a', b'\}$.

Now we address the final statement of the proposition, which concerns stable diffeomorphism.
Observe that~$\Z \cong H^2(N_{a,b}) \cong H^2(N_{a,b} \# S^{2k} \times S^{2k})$, and that the image of~$z$, which we call~$z_{st} \in H^2(N_{a,b} \# S^{2k} \times S^{2k})$, satisfies the equality~$z_{st}^{2k} = 2ab[N_{a,b} \# S^{2k} \times S^{2k}]$.  Since this property of $z_{st}$ and the fundamental class are preserved under diffeomorphism, it follows that if~$N_{a,b}$ and~$N_{a',b'}$ are stably diffeomorphic, then~$ab = a' b'$.

On the other hand, for a fixed product $ab = a'b'$, the manifolds $N_{a,b}$ and $N_{a',b'}$ are obtained from the $4k$-manifold $X_3$ by surgering away a stably hyperbolic form $H_{u, v}^\perp$. Recall that $u$ and $v$ depend on $a,b$, so in particular we may need to stabilise a different number of times for $H_{u, v}^\perp$ versus $H_{u', v'}^\perp$ to make them  hyperbolic. Let $h(u,v)$ and $h(u',v')$ be the number of stabilisations required, and let $h:= \max\{h(u,v),h(u',v')\}$.     Then for some $g$ we have
\[N_{a,b} \# W_g \cong X_3 \# W_h \cong N_{a',b'} \# W_g,\]
as desired.  So indeed $ab=a'b'$ if and only if $N_{a,b} \cong_{\text{st}} N_{a',b'}$.
\end{proof}

\section{$(4m{-}1)$-connected $8m$-manifolds with nontrivial homotopy stable class}
\label{s:3c8m-thm}

In this section, for every $m \geq 1$ we construct $(4m{-}1)$-connected $8m$-manifolds
with hyperbolic intersection form and with nontrivial homotopy stable class.
Specifically, we describe certain $8m$-manifolds $M_{a,b}$, for positive integers $a$ and $b$, and we will give bounds from above and below on the size of the homotopy stable class of $M_{a,b}$ in terms of $a$, $b$, and $m$. In particular, for each $m$ there are infinitely many choices of $a,b$ such that $|\HSC(M_{a, b})| >1$.

In contrast to the manifolds in the previous section, the homotopically inequivalent manifolds constructed here have isomorphic integral cohomology rings, but are not stably parallelisable.
We will detect that our manifolds are not homotopy equivalent using a refinement of the $m$th Pontryagin class.

This section is organised as follows.  In Section~\ref{subsection:pertinent-facts} we recall some facts about exotic spheres and the $J$ homomorphism, which we will need for the statement and the proof of Theorem~\ref{thm:4k-1-conn-8k-manifolds-examples}.  We state this theorem in Section~\ref{subsection:statement-of-theorem}.
In Section~\ref{section:almost-smooth-classification} we recall Wall's classification of $(4m{-}1)$-connected $8m$-manifolds up to the action of the group of homotopy $8m$-spheres, then in Section~\ref{ss:stable-classification} we determine the stable classification of such manifolds, again up to the action of the homotopy spheres.
Next, in Section~\ref{subsection:construction} we construct the manifolds $M_{a,b}$ appearing in Theorem~\ref{thm:4k-1-conn-8k-manifolds-examples} and we prove this theorem in Section~\ref{subsection:proof-of-thm-31}.

\subsection{Exotic spheres and the $J$-homomorphism}\label{subsection:pertinent-facts}

Let $\Theta_n$ denote the group of $h$-cobordism classes of \emph{homotopy $n$-spheres}, that is  closed, 
connected, oriented $n$-manifolds that are homotopy equivalent to $S^n$, with the group operation given by connected sum. By \cite{KeMi63} these are finite abelian groups.  We will briefly recall some of what is known about them, focussing on dimensions $n=8m$ and $n=8m{-}1$, for $m\geq 1$.

Recall that $bP_{n+1} \subseteq \Theta_n$ is the subgroup of $h$-cobordism classes of homotopy $n$-spheres which bound parallelisable
$(n{+}1)$-manifolds.  Kervaire and Milnor showed that this is a finite cyclic group, and for $n{+}1 = 4\ell > 4$ the order of $bP_{n+1}$ is given by a formula in terms of Bernoulli numbers and the image of the $J$-homomorphism~\cite{KeMi63}.
Following results of Adams~\cite{Adams:1966-1} and Quillen~\cite{Quillen71} on the $J$-homomorphism,
this formula led to the computation of $|bP_{4\ell}|$; we will give more details shortly.
The group $bP_{4\ell}$ is generated by the boundary of Milnor's $E_8$ plumbing \cite[V]{Bro72},
a $4\ell$-manifold obtained from plumbing disc bundles according to the $E_8$ lattice.

Let \[J_{n} \colon \pi_{n}(SO) \to \pi^s_{n}\] be the stable $J$-homomorphism~\cite[\S3]{Whitehead-defn-J},
where $\pi^s_{n}$ is the stable $n$-stem.
Kervaire and Milnor~\cite{KeMi63} showed that $\Theta_{8m} \cong \coker J_{8m}$ and that there is a short exact sequence
\[0 \to bP_{8m} \to \Theta_{8m-1} \to \coker J_{8m-1} \to 0.\]
Later Brumfiel~\cite{Brumfiel-I} defined a splitting $\Theta_{8m-1} \to bP_{8m}$ and so proved that
\[\Theta_{8m-1} \cong bP_{8m} \oplus \coker J_{8m-1}.\]

Consider a $(4m{-}1)$-connected $8m$-manifold $W$ with boundary $\partial W \in \Theta_{8m-1}$.
Extending work of Stolz~\cite{Stolz-highly-conn-mflds-and-their-boundaries} and
Burklund, Hahn and Senger~\cite{Burklund-Hahn-Senger},
Burklund and Senger~\cite[Theorem 1.2]{Burklund-Senger} proved that $[\partial W] \in bP_{8m}$, except possibly when $m=3$, when they also show that $2[\partial W] \in bP_{24}$.
%
For our purposes later in this section, we also assume that $W$ has signature $0$
and this ensures that $\partial W$ is a multiple of the homotopy sphere denoted $\Sigma_Q$ by
Krannich and Reinhold \cite[\S 2]{Krannich-Reinhold} (see just below Lemma~\ref{lem:Sigma_Q}
for the definition of $\Sigma_Q$.)

\begin{definition}
Let $\mathfrak{bp}_m$ be the order of $\Sigma_Q$ in $\Theta_{8m-1}$.
\end{definition}

\begin{remark}
The precise value of $\mathfrak{bp}_m$ can be calculated, assuming knowledge of the relevant Bernoulli numbers,
from \cite[Lemma~2.7]{Krannich-Reinhold}.
  In particular, $\mathfrak{bp}_m \mid |bP_{8m}|$. This is clear when $m \neq 3$, since $\Sigma_Q \in bP_{8m}$. It follows from a direct calculation when $m = 3$, given that the projection of $\Sigma_Q$ to $bP_{24}$ has order divisible by $2$.
\end{remark}


We now recall some facts about the $J$-homomorphism for context and later use.
We start with the stable
$J$-homomorphism $J_{4m-1} \colon \pi_{4m-1}(SO) \to \pi^s_{4m-1}$ and write
\[j_m := |\im(J_{4m-1})|.\]
%
For example
\[ j_1 = 24, \quad j_2 = 240, \quad \text{and} \quad j_3 =  504.\]
Later we will use the fact that $4 \mid j_m$, for $m = 1, 2$, as we see here.
Since the stable homotopy groups of spheres are finite, so is $j_m$.  Since $\pi_{4m-1}(SO) \cong \Z$, in fact $\im(J_{4m-1}) \cong \Z/j_m$.
By~\cite{Adams:1966-1} (see e.g.\ \cite[Theorem~6.26]{lueck-surgery-intro}), $j_m$ can be 
computed using the denominator of the rational number $B_m/4m$, where $B_m \in \Q$ is the $m$th Bernoulli number, defined by the generating function
\[\frac{e^t}{e^t-1} = 1 - \frac{t}{2} + \sum_{n = 1}^{\infty} \frac{(-1)^{n+1} B_n}{(2n)!} t^{2n}. \]
By \cite[Section~7]{KeMi63}, $|bP_{8m}|/(2^{4m-2} (2^{4m-1} -1))$ equals the numerator of the rational number $2B_{2m}/m$, from which one can compute $|bP_{8m}|$.   

Next we consider the unstable $J$-homomorphism, $J_{4m-1, 4m} \colon \pi_{4m-1}(SO_{4m}) \to \pi_{8m-1}(S^{4m})$,
which, along with the stable $J$-homomorphism, the Euler class $\e$ and the Hopf-invariant $H$, fits into the following commutative diagram with exact rows:
\begin{equation} \label{eq:e=HJ}\tag{*}
\xymatrix{
0 \ar[r] &
\pi_{4m-1}(SO_{4m}) \ar[r]^(0.45){\e \oplus S} \ar[d]^{J_{4m-1, 4m}} &
\Z \oplus \pi_{4m-1}(SO) \ar[d]^{\Id \oplus J_{4m-1}} \ar[r] &
\Z/2 \ar[d]^= \ar[r] &
0
 \\
0 \ar[r] &
\pi_{8m-1}(S^{4m}) \ar[r]^{H \oplus S}   &
\Z \oplus \pi^s_{4m-1} \ar[r] &
\Z/2 \ar[r] &
0
}
\end{equation}
The commutativity of left hand square in
\eqref{eq:e=HJ} is equivalent to the classical statements that $\e = H \circ J_{4m-1, 4m}$
and that the $J$-homomorphism commutes with stabilisation \cite[1.2 \& 1.3]{James-Whitehead-I}.
That $\e \oplus S$ is injective with index $2$ is reviewed in \cite[p.\ 171]{Wall-classification-n-1-conn-2n}.
That the same statements hold for $H \oplus S$ follows from
Toda's calculations in the exceptional cases $m = 1, 2$ \cite[V, (iii) \& (vii)]{Toda} and
from Adam's solution of the Hopf invariant $1$ problem for $m > 2$ \cite{Adams:1960}.
For $m > 2$, both $\e(\pi_{4m-1}(SO_{4m})) \subseteq \Z$ and $H(\pi_{8m-1}(S^{4m})) \subseteq \Z$
are index two subgroups and stabilisation is a split surjection, \cite{Bott-Milnor, Adams:1960}.
In particular this means that for $m > 2$
the Euler class $e$ is always even for rank $4m$ oriented vector bundles over $S^{4m}$.
When $m = 1, 2$, the maps $\e$ and $H$ are both onto and $\e = H \circ J_{4m-1,4m} \equiv S$~mod~$2$~\cite[p.~171]{Wall-classification-n-1-conn-2n} and $H \equiv S$~mod~$2$ by Toda's computations mentioned above.  These computations show that for $m=1$,
$H \oplus S \colon \pi_7(S^4) \cong \Z \oplus \Z/12 \to \Z \oplus \Z/24$ sends $(x,y) \mapsto (x,x+2y)$.
For $m=2$, the map $H \oplus S \colon \pi_{15}(S^8) \cong \Z \oplus \Z/120 \to \Z \oplus \Z/240$ is also given by $(x,y) \mapsto (x,x+2y)$. It follows that $H \equiv S$~mod~$2$ as asserted.

%
\subsection{Estimating $\HSC(M)$}\label{subsection:statement-of-theorem}
In this section we give upper and lower bounds for the homotopy stable class of
certain $(4m{-}1)$-connected $8m$-manifolds.  To state these bounds we require a certain amount
of notation.

Let $m$ be a positive integer and let $\{a, b\}$ be a pair of positive integers.
Since the dimensions $8$ and $16$ are
exceptional, we introduce the factor
\[
c_m :=
\begin{cases}
2 & m = \text{$1$ or $2$,} \\
1 & m > 2,
\end{cases} \]
to handle the exceptional dimensions.  We define
\[d := \gcd(a,b)c_m
\]
and write
\[ac_m = da' \text{ and } bc_m = db'\]
for some coprime $a',b'$.  Set
 \[A := a'b' = abc_m^2/d^2.\]
For a positive integer $n$ we let $\mathcal{P}_n$ be the set of prime factors of $n$:
\[\mathcal{P}_n :=  \{ p \in \mathbb{N}\, : \, p \text{ prime, } p \mid n\}. \]
We set $\ol j_m = j_m/\gcd(j_m, d)$ and consider the sets $\mathcal{P}_{\! A}$, $\mathcal{P}_{\ol j_m}$ and their intersection
\[\mathcal{P}_{\! A, m} :=   \mathcal{P}_{\! A} \cap \mathcal{P}_{\ol j_m},\]
the set of primes dividing both $\ol{j}_m$ and $A$.
We define the non-negative integers
\[q_{A} := |\mathcal{P}_{\! A}| - 1 \text{ and } q_{A,m} : = |\mathcal{P}_{\! A, m}| -1.\]
Now we can state the main theorem of this section.  Its proof will occupy the remainder of the section.

\begin{theorem}\label{thm:4k-1-conn-8k-manifolds-examples}
Let $m$ be a positive integer and let $\{a, b\}$ be a pair of positive integers such that $\mathfrak{bp}_{m} \mid ab$.
If $d = \gcd(a, b)$ and $\ol j_m = j_m/\gcd(j_m, d)$,
then the closed, $(4m{-}1)$-connected $8m$-manifolds $M_{a, b}$ constructed in Section~\ref{subsection:construction} satisfy the following:
\begin{enumerate}
\item\label{item-thm-4k-1-conn-8k-0} $M_{a,b}$ has hyperbolic intersection form,
\item\label{item-thm-4k-1-conn-8k-i} $|\SC(M_{a, b})/\Theta_{8m}| = 2^{q_{A}}$, and
\item\label{item-thm-4k-1-conn-8k-ii} $2^{q_{A, m}} \leq |\HSC(M_{a, b})| \leq \Big\lfloor \frac{\ol{j}_m^2 +2\ol{j}_m + 4}{4} \Big\rfloor$.
\end{enumerate}
\end{theorem}

Adam's work on $j_m$ \cite{Adams:1966-1},
a theorem of von Staudt and Clausen (see \cite[Theorem 3,~p.\ 233]{Ireland-Rosen}) on the
denominator of $B_m$, and a result of von Staudt on the numerator of $B_m$ (see
 \cite[Lemma 2]{Milnor1958}) combine to show that
 \[\mathcal{P}_{j_m} = \{ p \text{ prime} : (p-1) \mid 2m\}.\]
Since $2$ and $3$ certainly lie in the latter set, $|\mathcal{P}_{j_m}|  \geq 2$.
Now define \[q_m := |\mathcal{P}_{j_m}| - 1 \geq 1.\]
By choosing $a$ and $b$ with some care, we obtain the following corollary, which implies Theorem~\ref{thm:3-conn-8-manifolds-examples}.

\begin{corollary} \label{cor:4k-1-conn-8m-manifolds-examples}
Let $m$ be a positive integer
and let $\{a, b\}$ be a pair of positive, coprime integers such that $\mathfrak{bp}_{m} \mid ab$
and $j_m/c_m \mid A = abc_m^2$.
Then the closed, $(4m{-}1)$-connected $8m$-manifolds $M_{a, b}$ constructed in Section~\ref{subsection:construction} have hyperbolic intersection form and satisfy that $2 \leq 2^{q_m} \leq |\HSC(M_{a, b})|$.
\end{corollary}

In particular, any coprime, positive $a,b$ such that $\mathfrak{bp}_m \cdot j_m/c_m$ divides $A=abc_m^2$ satisfies the hypotheses of the corollary. Note that changing $A$ does not alter the lower bound, which is purely a function of $m$.

\begin{proof}
Since $a$ and $b$ are coprime, $d = c_m$, $\ol j_m = j_m/c_m$ and $\mathcal{P}_{\ol j_m} = \mathcal{P}_{j_m}$ (using $4\mid j_m$ for $m=1,2$).
Since $j_m /c_m = \ol j_m \mid A$ we see that $\mathcal{P}_{\ol j_m} \subseteq \mathcal{P}_{\! A}$ and therefore
$\mathcal{P}_{\! A, m} = \mathcal{P}_{\ol j_m} = \mathcal{P}_{j_m}$, so that $q_{A,m} = q_m$.
Since $q_m \geq 1$, the corollary follows from the lower bound in
Theorem~\ref{thm:4k-1-conn-8k-manifolds-examples}~\eqref{item-thm-4k-1-conn-8k-ii}.
\end{proof}

\subsection{The almost-diffeomorphism classification of $(4m{-}1)$-connected $8m$-manifolds}
\label{section:almost-smooth-classification}

In this section we recall the relevant part of Wall's classification results for closed, $(4m{-}1)$-connected $8m$-manifolds.
Recall that two closed manifolds $M_0$ and $M_1$ are \emph{almost diffeomorphic} if there
is a homotopy sphere $\Sigma$ and a diffeomorphism $f \colon M_0 \# \Sigma \to M_1$.

Let $M$ be a closed,  $(4m{-}1)$-connected $8m$-manifold, and equip $M$ with an orientation.
The {\em intersection form of $M$} is a symmetric bilinear form
\[ \lambda_M \colon H_{4m}(M) \times H_{4m}(M) \to \Z. \]
The {\em obstruction class} of $M$ is the homomorphism
\[ S\alpha_M \colon H_{4m}(M) \to \pi_{4m-1}(SO) \cong \Z \]
defined by representing a homology class $x$ by a smoothly embedded sphere $S^{4m}_x \hookrightarrow M$,
whose existence is ensured by Hurewicz theorem and \cite[Theorem~1(a)]{Haefliger}, and then taking the homotopy
class of the clutching map of the stable normal bundle of $S^{4m}_x$.
The map $S\alpha_M$ is the stabilisation of a map $\alpha_M$ defined by taking the
normal bundle of $S^{4m}_x$.
This will be important in the proof of Theorem~\ref{thm:class} below.
As shown by Wall \cite[p.\ 171~\&~Lemma~2]{Wall-classification-n-1-conn-2n}, if $m = 1, 2$ then the existence of rank $4m$ vector bundles over $S^{4m}$ with odd Euler class implies that the obstruction
class is characteristic for the intersection form; i.e.\ if $m=1$ or $2$ then for all $x \in H_{4m}(M)$
\begin{equation} \label{eq:ob_is_char}\tag{$\dag$}
\lambda_M(x, x) \equiv S\alpha_M(x)~\mathrm{mod}~2.
\end{equation}
For $m>2$, by Wall \cite[p.\ 171]{Wall-classification-n-1-conn-2n},
there is no relation between $S\alpha_M$ and $\lambda_M$. As also shown in \cite[p.\ 171~\&~Lemma~2]{Wall-classification-n-1-conn-2n}, since $\e = H \circ J_{4m-1,4m}$ and since for $m >2$ we have that $H \circ J_{4m-1,4m}$ is even, the Euler number is always even and therefore $\lambda_M(x,x) \equiv 0 \mod{2}$ for all $x \in H_{4m}(M)$.

For the homotopy classification, we consider the stable $J$-homomorphism
\[ J_{4m-1} \colon \pi_{4m-1}(SO) \to \Z/j_m \subseteq \pi^s_{4m-1}. \]
The {\em homotopy obstruction class} of $M$, $S\ah_M$, is the composition of $S\alpha_M$ with $J_{4m-1}$,
\[ S\ah_M := J_{4m-1} \circ S\alpha_M \colon H_{4m}(M) \to \Z/j_m. \]
Since $j_1$ and $j_2$ are divisible by $2$ the congruence of \eqref{eq:ob_is_char} implies that
if $m = 1, 2$ then
\begin{equation} \label{eq:hob_is_char}
\lambda_M(x, x) \equiv S\ah_M(x)~\mathrm{mod}~2.
\end{equation}

We now define the invariants we use to classify $(4m{-}1)$-connected $8m$-manifolds
up to almost diffeomorphism and homotopy equivalence.

\begin{definition}[Extended symmetric form]\label{defn:extended-symm-form}
Fix a homomorphism $v \colon G \to \Z/2$ from an abelian group $G$ to $\Z/2$.
An \emph{extended symmetric form over $v$} consists of a triple $(H,\lambda,p)$
where:
\begin{enumerate}[(1)]
  \item $H$ is a finitely generated free $\Z$-module;
  \item $\lambda \colon H \times H \to \Z$ is a symmetric, bilinear form; and
  \item\label{item:extended-symm-form-defn-3} $f \colon H \to G$ is a homomorphism such that $\lambda(x, x) \equiv v \circ f(x)$ mod~$2$.
    \end{enumerate}
Two extended symmetric forms $(H,\lambda, f)$ and $(H',\lambda',f')$ are equivalent if there is an isometry
$h \colon (H,\lambda) \to (H',\lambda')$ such that $f' \circ h = f \colon H \to G$.
\end{definition}

In our applications to $8m$-manifolds, the group $G$ will either be the infinite cyclic group $\pi_{4m-1}(SO) \cong \Z$ or the finite
cyclic group $\mathrm{Im}(J_{4m-1}) \cong \Z/j_m$.
Due to the existence of rank $4m$ bundles over $S^{4m}$ with odd Euler number when $m=1,2$, and the non-existence of such bundles for $ m\geq 3$, we set $v$ to be nonzero for $m = 1, 2$ (recall $2 \mid j_1$ and $2\mid j_2$) and zero for
$m > 2$.
Hence for $m>2$,~\eqref{item:extended-symm-form-defn-3} is just the requirement that $\lambda_M$ be even.  With these conventions on $v$, the following assignments define extended symmetric forms.

\begin{definition}[The extended symmetric forms of $M$]\label{defn:ext-symm-form-of-M}
Let $M$ be an oriented $(4m{-}1)$-connected $8m$-manifold.
\begin{enumerate}
  \item
The {\em smooth extended symmetric form} of $M$ is the triple
\[\bigl( H_{4m}(M),\lambda_M, S\alpha_M \bigr)
\]
with $G \cong \Z$.
\item
The {\em homotopy extended symmetric form} of $M$ is the triple
\[\bigl(H_{4m}(M),\lambda_M, S\ah_M \bigr)\]
with $G \cong \Z/j_m$.
\end{enumerate}
\end{definition}

The following result is a direct consequence of classification results of
Wall \cite[p.~170 \& Lemma~8]{Wall-classification-n-1-conn-2n}.

\begin{theorem}[Wall] \label{thm:class}
Let $M_1$ and $M_2$ be closed, oriented, $(4m{-}1)$-connected $8m$-manifolds.  The manifolds $M_1$ and $M_2$ are:
\begin{enumerate}
  \item\label{item:thm-class-1}  almost diffeomorphic, via an orientation-preserving diffeomorphism, if and only if their smooth extended symmetric forms are equivalent;
\item\label{item:thm-class-2}  homotopy equivalent, via a degree one homotopy equivalence, if and only if their homotopy extended symmetric forms are equivalent.
\end{enumerate}
\end{theorem}

When applying these classifications, we will later have to factor out by the effect of the orientation choice on the extended symmetric forms.

\begin{proof}
We start with the almost diffeomorphism classification~\eqref{item:thm-class-1}.
As mentioned above, the homomorphism $S\alpha_M$ is the stabilisation
of a certain quadratic form, the extended quadratic form of $M$,
which is the map
\[ \alpha_M \colon H_{4m}(M) \to \pi_{4m-1}(SO_{4m}), \]
defined by representing a homology class by a smoothly embedded sphere $S^{4m} \hookrightarrow M$,
and then taking the classifying map in $\pi_{4m-1}(SO_{4m})$ of the normal bundle of the embedded sphere.
For all $x, y \in H_{4m}(M)$, \cite[Lemma 2]{Wall-classification-n-1-conn-2n} (and the fact that $\e = H \circ J_{4m-1,4m}$) proves that $\alpha_M$ relates to the intersection form of $M$ by the equations
\[ \lambda_M(x, x) = \e(\alpha_M(x))
\quad \text{and} \quad
\alpha_M(x + y) = \alpha_M(x) + \alpha_M(y) + \lambda(x, y) \tau.\]
Here the map $\e \colon \pi_{4m-1}(SO_{4m}) \to \Z$ is the Euler number of the corresponding bundle and $\tau \in \pi_{4m-1}(SO_{4m})$
is the clutching function of the tangent bundle of $S^{4m}$.
Wall also proved~\cite[p.\ 170]{Wall-classification-n-1-conn-2n} that
the triple $(H_{4m}(M),\lambda_M,\alpha_M)$
is a complete almost diffeomorphism invariant of~$M$.
In fact, Wall stated his classification in terms of \emph{almost closed} manifolds: compact manifolds with boundary a homotopy sphere. But this also yields the almost diffeomorphism classification, as follows. If the extended symmetric forms of two closed $(4m{-}1)$-connected $8m$-manifolds are equivalent then by the almost closed classification the manifolds are diffeomorphic after removing a ball $D^{8m}$ from each. Gluing the balls back in compatibly with the diffeomorphism might change one of the manifolds by connected sum with a homotopy sphere, but nonetheless the two closed manifolds are almost diffeomorphic.  On the other hand almost diffeomorphic manifolds are diffeomorphic after removing a ball from each, and then by the classification the extended symmetric forms are equivalent.

As mentioned above, $S\alpha_M := S \circ \alpha_M$, where $S \colon \pi_{4m-1}(SO_{4m}) \to \pi_{4m-1}(SO)$ is
the stabilisation homomorphism.  The homotopy exact sequence of the
fibration $SO_{4m} \to SO_{4m+1} \to S^{4m}$ shows that the kernel of $S$ is generated by $\tau$~\cite[Lemma~1.3~and~Theorem~1.4]{LevineLectures} and since
\[ \e \oplus S \colon \pi_{4m-1}(SO_{4m}) \to \Z \oplus \pi_{4m-1}(SO) \]
is injective by \eqref{eq:e=HJ}
it follows that the pair $(\lambda_M(x, x), S\alpha_M(x)) = \bigl( \e(\alpha_M(x)), S(\alpha_M(x)) \bigr) \in \Z \oplus \pi_{4m-1}(SO) \cong \Z \oplus \Z$
determines $\alpha_M(x)$ for all $x \in H_{4m}(M)$.
The theorem now follows from Wall's almost diffeomorphism classification.

The proof of the homotopy classification is similar.  By Wall~\cite[Lemma 8]{Wall-classification-n-1-conn-2n},
the triple $(H_{4m}(M), \lambda_M, \alpha_M^h:= J_{4m-1, 4m} \circ \alpha_M)$
is a complete homotopy invariant of the manifolds under consideration.
Since $\e = HJ$ and
\[ H \oplus S \colon \pi_{8m-1}(S^{4m}) \to \Z \oplus \pi^s_{4m-1} \]
is injective by \eqref{eq:e=HJ}, it follows that the pair $(\lambda_M(x, x), S\ah_M(x)) = \bigl( H(\ah_M(x)), S(\ah_M(x)) \bigr) \in \Z \oplus \pi^s_{4m-1}$
determines $\ah_M(x)$ for all $x \in H_{4m}(M)$.
The theorem now follows from Wall's homotopy classification.
\end{proof}

\subsection{Stable almost-diffeomorphism classification of $(4m{-}1)$-connected $8m$-manifolds}\label{ss:stable-classification}
In this section we give the stable classification of closed $(4m{-}1)$-connected $8m$-manifolds up
to connected sum with homotopy $8m$-spheres.
Define the non-negative integer $d_M$ by the equation
\[ S\alpha_M(H_{4m}(M)) = d_M \Z.\]
Equivalently, $d_M$ is the divisibility of $S\alpha_M \in H^{4m}(M)$, where, since $M$ is $(4m{-}1)$-connected, we may regard $S\alpha_M$ as an element of the group $H^{4m}(M)$ via the inverse of the evaluation map $\operatorname{ev} \colon H^{4m}(M) \to \Hom(H_{4m}(M),\Z)$, which is an isomorphism.
In particular, it makes sense to consider the class $(S\alpha_M)^2 \in H^{8m}(M) \cong \Z$.

\begin{theorem}\label{thm:stable-diffeo-classification-8-manifolds}
Two closed, oriented, $(4m{-}1)$-connected $8m$-manifolds $M$ and $N$ with the same Euler characteristic are
almost stably diffeomorphic, via an orientation-preserving diffeomorphism, if and only if the following hold:
\begin{enumerate}
\item $d_M = d_N$,
\item $\sigma(M) = \sigma(N)$,
\item $\langle (S\alpha_M)^2, [M] \rangle = \langle (S\alpha_N)^2,[N]\rangle$.
\end{enumerate}
\end{theorem}

\begin{proof}
First, we note that $d_M$, the signature, and $(S\alpha_M)^2$ are invariants
of orientation preserving almost stable diffeomorphisms, so one implication holds.

For the other implication we assume that $M$ and $N$ are such that
$d_M = d_{N}$, $\sigma(M) = \sigma(N)$, and $(S\alpha_M)^2 = (S\alpha_N)^2$
and we show that $M$ and $N$ are stably diffeomorphic.
The normal $(4m{-}1)$-type of $M$ and $N$ is determined by $d = d_M = d_N$
and is described as follows.
Let $d$ be a non-negative integer.
Let $BO\an{4m{-}1} \to BO$ be the $(4m{-}1)$-connected cover of $BO$ and let
$p \in H^{4m}(BO\an{4m{-}1}) \cong \Z$ be a generator.  We regard $\rho_d(p)$, the mod~$d$
reduction of $p$, as a map $\rho_d(p) \colon BO\an{4m{-}1} \to K(\Z/d, 4m)$ and define $BO\an{4m{-}1, d_M}$
to be the homotopy fibre of $\rho_d(p)$.  The normal $(4m{-}1)$-type of $M$ and $N$ is represented by the fibration given by the composition
\[  BO\an{4m{-}1, d} \to BO\an{4m{-}1} \to BO.\]
For brevity, use $(B_d, \eta_d)$ to denote the fibration $\eta_d \colon BO\an{4m{-}1, d} \to BO$.
We assert that~$M$ and $N$ admit unique normal $(4m{-}1)$-smoothings $\ol \nu_M \colon M \to B_d$
and $\ol \nu_N \colon N \to B_d$.
We prove the assertion for $M$, as the proof for $N$ is identical.
Since $M$ is $4m$-connected, its stable normal bundle $\nu_M \colon M \to BO$ lifts (up to homotopy) uniquely  to $\nu_{4m} \colon M \to BO\langle 4m \rangle$.
In order to lift $\nu_{4m}$ to~$B_d$, we consider the long exact sequence (of pointed sets) of the fibration
\[  0=H^{4m-1}(M;\Z/d) \to [M,B_d] \to [M,BO\langle 4m\rangle] \to H^{4m}(M;\Z/d)  \to \cdots, \]
where on the left, we used $[M,\Omega K(\Z/d,4m)]=[M, K(\Z/d,4m-1)]=H^{4m-1}(M;\Z/d)=0$, because~$M$ is $4m$-connected.
The assertion is now proved by noting that $\nu_{4m} \in [M,BO\langle 4m \rangle]$ maps to $S\alpha_M \in H^{4m}(M;\Z/d)$, which is zero by definition of the divisibility $d_M$.

By \cite[Theorem 2]{KreckSurgeryAndDuality}, $M$ and $N$ are
orientation preserving stably diffeomorphic if
\[ [M, \ol \nu_M] = [N, \ol \nu_N] \in \Omega_{8m}(B_d, \eta_d).\]
Since homotopy $8m$-spheres have a unique $(B_d, \eta_d)$-structure, there
is a well-defined homomorphism $\Theta_{8m} \to \Omega_{8m}(B_d, \eta_d)$.
Now the arguments in Wall's computation of the Grothendieck groups of almost closed $(4m{-}1)$-connected $8m$-manifolds \cite[Theorem 2]{Wall-classification-n-1-conn-2n} show that there
is an exact sequence
\begin{equation} \label{eq:Omega}\tag{$\Omega$}
\Theta_{8m} \to \Omega_{8m}(B_d, \eta_d) \xra{(\sigma,\, (S\alpha)^2)} \Z^2,
\end{equation}
where $\sigma([M, \ol \nu_M]) = \sigma_M$ and $S\alpha^2([M, \ol \nu_M]) = (S\alpha_M)^2([M])$.
It follows that there is a homotopy $8m$-sphere $\Sigma$ such that
$[M \# \Sigma, \ol \nu_{M \# \Sigma}] = [N, \ol \nu_N] \in \Omega_{8m}(B_d, \eta)$.
Hence $M \# \Sigma$ and $N$ are stably diffeomorphic and so
$M$ and $N$ are almost stably diffeomorphic.
\end{proof}

\subsection{Construction of the manifolds $M_{a, b}$}\label{subsection:construction}

In this section we construct the manifolds~$M_{a,b}$ appearing in Theorem~\ref{thm:4k-1-conn-8k-manifolds-examples}.
Let $a$ and $b$ be positive integers such that $\mathfrak{bp}_m \mid ab$.
We will build simply-connected, closed $8m$-manifolds $M_{a,b}$ with the cohomology ring of $S^{4m} \times S^{4m}$ by attaching handles to an $8m$-ball.
We attach two $4m$-handles $h_x$ and $h_y$, diffeomorphic to $D^{4m}\times D^{4m}$, to $D^{8m}$ using attaching maps $\phi_x,\phi_y \colon S^{4m-1} \times \{0\} \to S^{8m-1}$ with linking number~$1$. Note that for $m \geq 1$, 2-component links $S^{4m-1} \sqcup S^{4m-1} \hookrightarrow S^{8m-1}$ are classified up to smooth isotopy by the linking number, an integer~\cite[Theorem~in~Section 5]{HaefligerDifferentiableLinks}.
There is more data needed for the attaching maps, which for each $4m$-handle corresponds to a choice of framing for the attaching sphere $S^{4m{-}1} \subseteq S^{8m-1}$.
The framings that induce a given orientation are in one to one correspondence with homotopy classes of maps $[S^{4m{-}1},SO_{4m}]$ where the class of the constant map corresponds to the framing
which extends over an embedded $4m$-disc $D^{4m} \subseteq S^{8m-1}$.
Recall from \eqref{eq:e=HJ} that $\pi_{4m{-}1}(SO_{4m}) \cong \Z \oplus \Z$, detected by $\e \oplus S$ (although this map is not an isomorphism).
We are attaching $4m$-handles~$h_x$ $h_y$; let~$x$ and~$y$ denote the corresponding classes in $(4m)$th homology
and let $\xi_x,\xi_y  \in \pi_{4m-1}(SO_{4m})$ be the framings for the attaching maps.

Since we want $\lambda(x,x)=0$, we require that $\e(\xi_x) = 0$ but we are otherwise free to choose~$\xi_x$.  Recall that $c_m = 2$ if $m = 1, 2$ and $c_m = 1$ if $m \!> \! 2$,
%
%
fix an isomorphism $\pi_{4m-1}(SO) = \Z$ and choose $\xi_x$ such that $S(\xi_x) = a c_m$.
By the discussion following \eqref{eq:e=HJ}, we can find such a~$\xi_x$ for any choice of $a$.
Similarly, we attach the handle $h_y$ with $\e(\xi_y) = 0$ and $S(\xi_y) = b c_m$. Again, we can find such a $\xi_y$ for any $b$.
After attaching the pair of $4m$-handles, we write $W := W_{a, b}$ for the resulting compact
$8m$-manifold with boundary.
Note that there is a homotopy equivalence $W \simeq S^{4m} \vee S^{4m}$.
As above let $x$ and $y$ be generators of $\Z^2 \cong H_{4m}(W)$ and let $\{x^*, y^*\}$ be the dual basis for
$\Z^2 \cong H^{4m}(W) = H_{4m}(W)^*$.
The manifold $W=W_{a, b}$ has smooth extended symmetric form given by
\[ \bigl( H_{4m}(W), \lambda_W, S\alpha_W \bigr) = \left( \Z^2, \twotwo{0}{1}{1}{0}, \cvec{a c_m}{b c_m} \colon \Z^2 \to \Z \right),\]
where the notation for $S\alpha_W$ means that $S\alpha_W(x) = ac_m$ and $S\alpha_W(y) = bc_m$.

Alternatively, the construction thus far can be achieved by taking the two $D^{4m}$-bundles over
$S^{4m}$ determined by $\xi_{x}$ and $\xi_{y}$, and plumbing them together once.

The boundary of $W_{a, b}$ is a homotopy $(8m{-1})$-sphere, which we denote by $\Sigma_{a, b}$.
In particular, $\del W_{1, 1}$ is by definition the homotopy sphere $\Sigma_Q$ from \cite[\S 2]{Krannich-Reinhold}.
More generally, $\del W_{a,b}$ is given as follows.

\begin{lemma} \label{lem:Sigma_Q}
$[\del W_{a, b}]  = [ab \Sigma_Q] \in \Theta_{8m-1}$.
\end{lemma}

\begin{proof}
Recall from \cite[\S 17]{Wall-n-1-conn-2n+1}
the group $A^{\an{4m}}_{8m}$ of bordism classes of $(4m{-}1)$-connected
$8m$-manifolds with boundary a homotopy sphere, where the bordisms are required to be $h$-cobordisms on the boundary. Addition is via boundary connected sum. Taking the boundary defines a homomorphism $A^{\an{4m}}_{8m} \to \Theta_{8m-1}$,
and the characteristic numbers $\sigma$ and $(S\alpha)^2$ of \eqref{eq:Omega} are
also well-defined on $A^{\an{4m}}_{8m}$.
Indeed, Wall \cite[Theorems 2 \& 3]{Wall-classification-n-1-conn-2n}
proved that $\sigma \oplus (S\alpha)^2 \colon A^{\an{4m}}_{8m} \to \Z^2$ is an injective homomorphism.
Since $W_{a, b}$ satisfies $\sigma(W_{a, b}) = 0$ and $(S\alpha_{W_{a, b}})^2 = 2abc_m^2$,
we have that $(S\alpha_{W_{1, 1}})^2 = 2c_m^2$ and so
\[(S\alpha_{W_{a, b}})^2 = 2abc_m^2 = ab (S\alpha_{W_{1,1}})^2 = (S\alpha_{\natural^{ab} W_{1,1}})^2,\]
where the last equality used that $(S\alpha)^2$ is a homomorphism.
Since $\sigma \oplus (S\alpha)^2$ is injective, $W_{a,b} = \natural^{ab} W_{1,1} \in A^{\an{4m}}_{8m}$.  So $\partial W_{a,b}$ and  $\partial (\natural^{ab} W_{1,1}) = ab \Sigma_Q$ are $h$-cobordant and therefore diffeomorphic.
\end{proof}


From Lemma \ref{lem:Sigma_Q} and our assumption that $\mathfrak{bp}_m \mid ab$,
it follows that $[\Sigma_{a, b}] = 0 \in bP_{8m}$, so that there is a choice of
diffeomorphism $f \colon \Sigma_{a, b} \to S^{8m-1}$.
We write $M_{a, b, f} = W_{a, b} \cup_f D^{8m}$ for the closure of $W_{a, b}$ built using a diffeomorphism
$f \colon \Sigma_{a, b} \to S^{8m-1}$.
We will also use $M_{a, b}$ to ambiguously denote any $M_{a, b, f}$.
For any other choice of diffeomorphism $f'$, $M_{a,b,f}$ and $M_{a,b,f'}$ are almost diffeomorphic.

Let us record the values of the key invariants on $M_{a,b}$.
The stable almost diffeomorphism invariants of $M_{a,b}$ are $d_{M_{a,b}}= \mathrm{gcd}(a, b)c_m$,
$\sigma(M_{a,b}) = 0$, and $(S\alpha_{M_{a,b}})^2 = 2abc_m^2$.
The extended symmetric form of $M: =M_{a,b}$ is the same as that of $W$:
\[ \bigl( H_{4m}(M), \lambda_M, S\alpha_M \bigr) = \left( \Z^2, \twotwo{0}{1}{1}{0}, \cvec{a c_m}{b c_m} \colon \Z^2 \to \Z \right).\]
This completes the construction of the manifolds $M_{a,b}$.

\subsection{The proof of Theorem \ref{thm:4k-1-conn-8k-manifolds-examples}}
\label{subsection:proof-of-thm-31}

Now that we have constructed the $(4m{-}1)$-connected $8m$-manifolds $M_{a,b}$, we are ready to prove Theorem \ref{thm:4k-1-conn-8k-manifolds-examples}.

\begin{proof}[Proof of Theorem \ref{thm:4k-1-conn-8k-manifolds-examples}]
Let $a$ and $b$ be positive integers such that $\mathfrak{bp}_m \mid ab$.
By construction the oriented manifolds $M_{a,b}$ have hyperbolic intersection form, so Theorem \ref{thm:4k-1-conn-8k-manifolds-examples}~\eqref{item-thm-4k-1-conn-8k-0} is immediate.

As before write $d:= \gcd(a,b)c_m$ and define $A:= abc_m^2/d^2$.
Let $p_{1},\dots,p_{q_A+1}$ be the prime-power factors of $A$, which are powers of pairwise distinct primes.
Then there are $2^{q_A}$ ways to express $A$ as a product $y_iz_i$ of coprime positive integers,
 counting unordered pairs  $\{y_i,z_i\}$.
We consider the $8m$-manifolds
\[\{M_i := M_{dy_i,dz_i}\}_{i=1}^{2^{q_A}}.\]
For each $i$, $d_{M_i} = d$, $\sigma(M_i) =0$, and
$\langle (S\alpha_{M_i})^2,[M_i]\rangle = 2dy_idz_i= 2d^2 A = 2abc_m^2$.
Therefore the manifolds $M_i$ are pairwise almost stably diffeomorphic
 by Theorem~\ref{thm:stable-diffeo-classification-8-manifolds}.
 A priori they could not all lie in $\SC(M_{a, b})$, but the ambiguity of whether they are actually stably diffeomorphic can be removed by more carefully choosing the diffeomorphisms $f_i \colon \Sigma_i \to S^{8m-1}$ used to glue on $D^{8m}$ in the construction of the $M_i$. By changing the choice of the identification $f_i$ we can change $M_i$ by connected sum with an exotic sphere of our choice.
The $M_i$ were only determined up to this choice in our construction, so let us assume we have made this consistently so that $M_i \in \SC(M_{a, b})$ for every $i=1,\dots,2^{q_{A}}$.  In $\SC(M_{a, b})/\Theta_{8m}$ this choice of the $f_i$ is in any case irrelevant.

When we discuss extended symmetric forms on $\Z^2$, we will always mean with respect to a particular choice of basis.  For $M_i$, with its fixed choice of fundamental class $[M_i]$, we shall use a basis with respect to which the intersection form is represented by $H^+(\Z) = \bsm 0 & 1 \\ 1 & 0 \esm$.
We have constructed the $M_i$ so that with respect to such a basis $S\alpha_{M_i} \colon \Z^2 \to \Z$ is represented by $\bsm  ac_m \\bc_m \esm$ with $ab>0$ and $c_m \in \{1,2\}$.

The smooth extended symmetric forms of the $M_i$ are pairwise distinct, since isometries of the rank two hyperbolic intersection form can only change the sign and permute the basis elements.
The map $S\alpha_{M_i} \colon \Z^2 \to \Z$ is given by $(dy_i,dz_i)$.  Since the unordered pairs $\{dy_i,dz_i\}$ are pairwise distinct,  by the almost diffeomorphism classification of Theorem~\ref{thm:class}~\eqref{item:thm-class-1}, the $M_i$ are pairwise distinct up to orientation-preserving almost diffeomorphism.
We will be able to deduce that $|\SC(M_{a, b})/\Theta_{8m}| \geq 2^{q_{A}}$ once we have factored out by the effect of the choice of orientation of the $M_i$. In other words we must show that there are also no orientation-reversing almost diffeomorphisms from $M_i$ to $M_j$, for $i \neq j$, or equivalently that there is no orientation-preserving diffeomorphism $M_i \cong -M_j$.

Changing the orientation of $M_j$ changes the smooth extended symmetric form (with respect to the same basis for $H_{4m}(M)$) by altering the sign of the intersection form, but does not affect $S\alpha_{M_j}$.  To see this, note that while changing the orientation of $M_j$ changes the induced orientation of the fibres of the normal bundle of an embedded sphere $x$,  $S\alpha_{M_j}(x) \in \pi_{4m-1}(SO)$ is the clutching map of this normal bundle, and this is unaffected by the orientation of the fibres.

The isometries from the rank 2 hyperbolic form $H^+(\Z) = \bsm 0 & 1 \\ 1 & 0 \esm$
 to its negative $-H^+(\Z)$ consist of the self-isometries of the hyperbolic form, namely $\pm \Id$ and $\bsm 0 & 1 \\ 1 & 0 \esm$, composed with either $\bsm 1 & 0 \\ 0 & -1 \esm$ or $\bsm -1 & 0 \\ 0 & 1 \esm$.
Thus an orientation-reversing almost diffeomorphism could identify the smooth extended symmetric form characterised by $(H^+(\Z),\pm \{v,w\})$ with one of the extended symmetric forms $(-H^+(\Z),\pm\{-v,w\})$ or $(-H^+(\Z),\pm\{v,-w\})$.  But for both $M_i$ and $-M_i$, the corresponding pair of integers is $\pm\{v,w\} = \pm \{dy_i,dz_i\}$, where both elements have the same sign. So our manifolds $\{M_i\}$ are indeed distinct up to almost diffeomorphism.
This proves that $|\SC(M_{a, b})/\Theta_{8m}| \geq 2^{q_{A}}$.

Next we prove that $|\SC(M_{a, b})/\Theta_{8m}| \leq 2^{q_{A}}$. Any closed $8m$-manifold $M$ that is almost stably diffeomorphic to $M_{a,b}$ is also necessarily $(4m{-}1)$-connected, the divisibility of $S\alpha_M$ is~$d$, and the intersection form is rank 2, indefinite, and even, and therefore either hyperbolic or $-H^+(\Z)$.
If $M$ and $M_{a,b}$ are almost stably diffeomorphic then there is an orientation on $M$ such that $M$ and $M_{a,b}$ are almost stably diffeomorphic via an orientation-preserving stable diffeomorphism.  Use this orientation, and choose a basis for $H_{4m}(M)$ with respect to which the intersection form of $M$ is $H^+(\Z)$.
Observe that the manifolds $M_i$ cover all possibilities for $S\alpha_M$ while keeping $(S\alpha_M)^2$ a fixed multiple of the dual fundamental class. (If $S\alpha_M = \bsm -ac_m \\ bc_m \esm$, for example, then $(S\alpha_M)^2 = -2abc_m^2 <0$, whereas $(S\alpha_{M_{a,b}})^2 = 2abc_m^2 >0$.
This would contradict that $M_{a,b}$ and $M$ are orientation-preserving almost stably diffeomorphic.)
It follows by Theorem~\ref{thm:class}~\eqref{item:thm-class-1} that every such $M$ is almost-stably diffeomorphic to one of the~$M_i$, and therefore  $|\SC(M_{a, b})/\Theta_{8m}| \leq 2^{q_{A}}$ as desired.
This completes the proof of Theorem~\ref{thm:4k-1-conn-8k-manifolds-examples}~\eqref{item-thm-4k-1-conn-8k-i}.

To prove~\eqref{item-thm-4k-1-conn-8k-ii}, we need to estimate the size of the homotopy stable class of $M_{a,b}$ from above and below.
We begin with the upper bound.
As above, every closed $8m$-manifold $M$ stably diffeomorphic to $M_{a,b}$ has an orientation such that $M$ has hyperbolic intersection form and $d_M=d$.
The possibilities for~$S\alpha_M^h$, up to equivalence of extended symmetric forms, are therefore given by an unordered pair of elements of $\Z/j_m$, both of which are divisible by~$d$.
 Such an element of~$\Z/j_m$ lies in the subgroup generated by $\gcd(j_m, d)$, and so there are~$\ol j_m = j_m/\gcd(j_m, d)$ possibilities.
We assert that there are
\[\lmfrac{\ol j_m (\ol j_m + 1)}{2} \]
such pairs.
To see this, there are ${\ol{j}_m \choose 2} = \tmfrac{\ol j_m (\ol j_m - 1)}{2}$ choices with distinct elements $(x,y)$, and $\ol{j}_m$ choices of the form $(x,x)$. Then
$\smfrac{\ol j_m (\ol j_m - 1)}{2} + \ol{j}_m = \tmfrac{\ol j_m (\ol j_m + 1)}{2}$, which is the count asserted.
Next, we also factor out by the action of $\Z/2$ on our set of unordered pairs which multiplies both numbers by $-1$. In the case that $\ol{j}_m$ is even, there are $\tmfrac{\ol{j}_m}{2} {+} 1$ fixed points of this action of the form $(x,-x)$, and also $(0,\tmfrac{\ol{j}_m}{2})$ is a fixed point. Thus there are precisely $\tmfrac{\ol{j}_m}{2} +2$ fixed points of a $\Z/2$ action on a set with $\tmfrac{\ol j_m (\ol j_m + 1)}{2}$ elements. A short calculation then shows that there are
\[\lmfrac{\ol{j}_m^2 +2\ol{j}_m + 4}{4}\]
orbits.  A similar calculation for $\ol{j}_m$ odd gives
\[\lmfrac{(\ol{j}_m + 1)^2}{4} = \Big\lfloor \lmfrac{\ol{j}_m^2 +2\ol{j}_m + 4}{4} \Big\rfloor\]
orbits. The right hand side is equal for both parities of $\ol{j}_m$, and   gives our desired upper bound.
Note that this upper bound does not take into account the requirement for the product $ab$ to be constant within a stable diffeomorphism class.

It remains to prove that $2^{q_{A, m}} \leq |\HSC(M_{a, b})|$.
As above let $p_{1},\dots,p_{q_A+1}$ be the prime-power factors of $A$, which are powers of pairwise distinct primes.
By reordering if necessary, assume that $p_1,\dots,p_{q_{A,m}+1}$ are the prime-powers of the form $p^\ell$ where
$p \mid \ol j_m$.
(It could be that the highest exponent of $p$ that divides $\ol j_m$ is less than the highest exponent of $p$ that divides $A$.)
Recall that $d = \gcd(a,b) c_m$ and write
\[d' := d \cdot \prod_{\iota=q_{A,m}+2}^{q_{A} +1} p_\iota \text{ and } A' := \prod_{\iota=1}^{q_{A,m} +1} p_\iota.\]
Note that $d'A' =dA$.
There are $2^{q_{A,m}}$ essentially distinct ways to express $A'$ as a product $v_iw_i$ of coprime positive integers, counting unordered pairs $\{v_i,w_i\}$.
We consider the $8m$-manifolds
\[\{M_i := M_{dv_i,d'w_i}\}_{i=1}^{2^{q_{A,m}}}.\]
For each $i$, $d_{M_i} = d$, $\sigma(M_i) =0$, and
$\langle (S\alpha_{M_i})^2,[M_i]\rangle = 2dv_id'w_i = 2dd' A' = 2d^2 A = 2abc_m^2$.
Therefore the $M_i$ are pairwise almost stably diffeomorphic by Theorem~\ref{thm:stable-diffeo-classification-8-manifolds}, so up to homotopy equivalence they are all stably diffeomorphic.  As above, the ambiguity of whether they are actually stably diffeomorphic can be removed by more carefully choosing the diffeomorphisms $f_i \colon \Sigma \to S^{8m-1}$ used to glue on $D^{8m}$ in the construction of the $M_i$.
Let us assume once again that we have made this choice consistently so that  $M_i \in \HSC(M_{a, b})$ for every $i=1,\dots,2^{q_{A,m}}$.


Next we show that the $M_i$ are distinct up to homotopy equivalence. For this, by Theorem~\ref{thm:class} we need to distinguish their homotopy extended symmetric forms, by showing that the maps $S\alpha_{M_i}^h \colon \Z^2 \to \Z/j_m$ are pairwise distinct, up to precomposing with an isometry of the hyperbolic form, or to allow for the possibility of an orientation-reversing homotopy equivalence, up to an isometry between the hyperbolic form and its negative.  This means we have to show that the unordered pair of elements of $\Z/j_m$ determining $S\alpha_{M_i}^h$ and $S\alpha_{M_j}^h$ are distinct up to changing signs.

Let $M_i$ and $M_j$ be two of our manifolds, for $i \neq j$. We will show that they are not homotopy equivalent.
First, $\gcd(d,j_m)$ divides $d$, so divides $dv_i$ and $d'w_i$. As above write $\ol{j}_m := j_m/\gcd(d,j_m)$. The map $S\alpha_{M_i}^h \colon \Z^2 \to \Z/j_m$ factors as
\[S\alpha_{M_i}^h \colon \Z^2 \to \Z/\ol{j}_m \to \Z/j_m\]
for all $i$, where $\Z/\ol{j}_m \to \Z/j_m$ is the standard inclusion sending $1 \mapsto \gcd(d,j_m)$. Define
\[\ol{d} := \smfrac{d}{\gcd(d,j_m)} \text{ and } \ol{d}' := \smfrac{d'}{\gcd(d,j_m)} = \ol{d} \cdot \prod_{\iota=q_{A,m}+2}^{q_{A} +1} p_\iota.\]
We obtain
\[S\ol{\alpha}^h_{M_i} = \begin{pmatrix}
  \ol{d} v_i \\ \ol{d}'w_i
\end{pmatrix} \colon \Z^2 \to \Z/\ol{j}_m.\]
It suffices to prove that for $i \neq j$ the resulting pairs $\{\ol{d}v_i,\ol{d}'w_i\}$ and $\{\ol{d}v_j,\ol{d}'w_j\}$ are distinct, up to signs and switching the orders.
Note that $\gcd(\ol{d},\ol{j}_m) =1 = \gcd(\ol{d}',\ol{j}_m)$.

Let $p$ be a prime dividing $A'$.  Up to possibly changing the orders of $v_i$ and $w_i$, and of $v_j$ and $w_j$, assume that $p$ divides $v_i$ and $v_j$.
If so, $p$ does not divide $w_i$ and $w_j$, since $\gcd(v_i,w_i) = 1 = \gcd(v_j,w_j)$.

Now let $q \neq p$ be a prime dividing $A'$ such that either:
\begin{enumerate}[(i)]
  \item $q$ divides $w_j$ but $q$ does not divide $w_i$; or
  \item $q$ divides $w_i$ but $q$ does not divide $w_j$.
\end{enumerate}
Without loss of generality suppose that (i) holds.  Then also $q$ divides $v_i$ but~$q$ does not divide~$v_j$, since both pairs $(v_i,w_i)$ and $(v_j,w_j)$ are coprime.
There exists such a $q$, unless $q_{A,m} =0$, in which case $2^{q_{A,m}} =1$ and we have nothing to prove anyway. So we can assume that $q_{A,m}$ is positive and that such a $q$ exists.
The idea is that the primes $p$ and $q$ are chosen so that they divide the same element of the unordered pair associated with the homotopy extended symmetric form for $M_i$, but divide different elements of the unordered pair for $M_j$. It is this distinction we want to detect.

We consider the images of the four elements $\ol{d}v_i$, $\ol{d}'w_i$, $\ol{d}v_j$, and $\ol{d}'w_j$ of $\Z/\ol{j}_m$ under the canonical surjections
%
\[\rho_p \colon \Z/\ol{j}_m \to \Z/p \text{ and } \rho_q \colon \Z/\ol{j}_m \to \Z/q.\]
Since $p$ and $q$ divide $\ol{j}_m$ and $\gcd(\ol{d},\ol{j}_m)=1 = \gcd(\ol{d}',\ol{j}_m)$, we know that $p$ and $q$ do not divide $\ol{d}$ and do not divide $\ol{d}'$.
Therefore for the $\Z/p$ reductions we have
\[\rho_p(\ol{d}v_i) = 0,\;  \rho_p(\ol{d}'w_i) \neq 0,\;   \rho_p(\ol{d}v_j) = 0, \text{ and } \rho_p(\ol{d}'w_j) \neq 0, \]
while for the $\Z/q$ reductions we have:
\[\rho_q(\ol{d}v_i) = 0,\;  \rho_q(\ol{d}'w_i) \neq 0,\;   \rho_q(\ol{d}v_j) \neq 0, \text{ and } \rho_q(\ol{d}'w_j) = 0. \]
We indicate one of these calculations briefly, that $\rho_p(\ol{d}'w_i) \neq 0$, to give the idea. If $\ol{d}'w_i$ were $0$ modulo $p$ then for some $a,b \in \Z$ we would have $ap = \ol{d}'w_i + b\ol{j}_m \in \Z$. But $p \mid \ol{j}_m$ and so $p$ divides $\ol{d}'w_i$, which is a contradiction.

Note that switching the sign of an element in $\Z/\ol{j}_m$ preserves whether or not its image under $\rho_p$ or $\rho_q$ is zero.
Let us summarise the calculations above. For $\{\ol{d}v_i,\ol{d}'w_i\}$, one element is zero under the reductions modulo $p$ and $q$, while the other element is nonzero under both reductions.
On the other hand, for the pair $\{\ol{d}v_j,\ol{d}'w_j\}$ we have shown that precisely one element is zero under each of the modulo $p$ and modulo $q$ reductions.
Switching the orders of the elements and switching signs preserves these descriptions, and therefore~$M_i$ and ~$M_j$ are not homotopy equivalent.
It follows that $|\HSC(M_{a, b})|$ is at least $2^{q_{A, m}}$, as desired.
\end{proof}

\section{\texorpdfstring{spin$^{c}$\hspace{-1pt}}{Spin-c} structures on 4-manifolds}\label{section:spin-c-structures}

As explained in the introduction, the homotopy stable class is trivial for every closed, simply-connected 4-manifold.  However a parallel phenomenon occurs when one considers equivalence classes of spin$^{c}$ structures on the tangent bundle. In this section we illustrate this on $S^2 \times S^2$.

For all $n \geq 2$, the group $\Spin_n$ is the connected double cover of $SO_n$ and the group
$\Z/2$ acts by deck transformations.
The group $\Z/2$ acts on $U(1) \cong S^1$ by complex conjugation.  We quotient out by the diagonal action on the product to obtain:
\[\Spinc_n := U(1) \times_{\Z/2} \Spin_n\]
There are well-defined maps
\[\xymatrix{  U(1) & \Spinc_n \ar[r]^(0.5){\pr_2} \ar[l]_(0.45){\pr_1} &  SO_n  }\]
obtained as the composition of the double cover $\Spinc_n \to U(1) \times SO_n$ with the projections.

There are natural inclusions $\Spinc_n \hookrightarrow \Spinc_{n+1}$ and
the stable spin$^c$ group is defined by $\Spinc := \colim_{n \to \infty} \Spinc_n$.
There are also stable projections
\[\xymatrix{  U(1) & \Spinc \ar[r]^(0.5){\pr_2} \ar[l]_(0.45){\pr_1} &  SO,}\]
where $SO$ is the stable special orthogonal group.
We will use the same notation $\pr_1$, $\pr_2$  for the induced maps on classifying spaces.

\begin{definition}
\label{def:Spinc}
%
%
Let $M$ be a closed, oriented $n$-manifold.
A \emph{spin$^{c}$\hspace{-1pt} structure} on $M$ is a lift
\[\xymatrix{& B\Spinc \ar[d]^{\pr_2} \\ M \ar@{-->}[ur]^{\s} \ar[r]^(0.45){\tau_M} & BSO}\]
of the stable tangent bundle's classifying map to $B\Spinc$.
\end{definition}
\noindent
For more background on spin$^{c}$\hspace{-1pt} structures on 4-manifolds, we refer to e.g.\ \cite[Section~2.4.1]{GompfStip} and \cite[Sections~10.2~\&~10.7]{Scorpan}.

\begin{lemma}[{\cite[Prop.~2.4.16]{GompfStip}}]\label{lemma:spin-c-exists}
Every oriented $4$-manifold admits a spin$^{c}$\hspace{-1pt} structure.
\end{lemma}

\begin{proof}
In~\cite{GompfStip}, spin$^{c}$ structures on $4$-manifolds are defined by using $B\Spinc_4$ in place of $B\Spinc$, and~\cite[Prop.~2.4.16]{GompfStip} proves the existence of a lift of the classifying map of the (unstable) tangent bundle to~$B\Spinc_4$.
Composing with the maps to the colimit, this implies the existence of a spin$^{c}$-structure in the sense of Definition~\ref{def:Spinc}.
\end{proof}

\begin{definition}[Equivalence of spin$^c$ structures]\label{defn:equiv-of-spinc}
Let $M$ be a closed, oriented $4$-manifold.
\begin{enumerate}
\item\label{item:defn-equiv-spinc} Two spin$^{c}$\hspace{-1pt} structures $\s_1$ and $\s_2$ on $M$ are \emph{equivalent} if there is an orientation-preserving diffeomorphism $f \colon M \to M$
such that $\s_1, \s_2 \circ f \colon M \to B\Spinc$ are homotopic over $BSO$; i.e.\ there is homotopy $K$
and a commutative diagram
\[\xymatrix{& B\Spinc \ar[d]^{\pr_2} \\ M \times I \ar[ur]^K \ar[r]^{\tau_{M \times I}} & BSO}\]
where $K$ restricts to $\s_1$ on $M \times \{0\}$ and $\s_2 \circ f$ on $M \times \{1\}$.
\item Two spin$^{c}$\hspace{-1pt} structures $\s_1$ and $\s_2$ are \emph{homotopic} if they are equivalent as in the previous item, with $f=\Id_M$.
\end{enumerate}
\end{definition}



Recall that the projection onto the first component gives a compatible collection of maps $\pr_1 \colon B\Spinc_n \to BU(1)$, $n \in \mathbb{N}$. Therefore, passing to the colimit and keeping the same notation, we obtain a map $\pr_1 \colon B\Spinc \to BU(1)$.

\begin{definition}
Via the map $\pr_1 \colon B\Spinc \to BU(1)$, a spin$^{c}$\hspace{-1pt} structure $\s \colon M \to B\Spinc$ on a 4-manifold $M$ determines a line bundle $\mathcal{L}_\s$.  The \emph{first Chern class} of $\s$ is defined by 
\[c_1(\s) := c_1(\mathcal{L}_\s) \in H^2(M).\]
\end{definition}

Noting that $BU(1)$ is a $K(\Z,2)$, $c_1(\s)$ corresponds to $\pr_1 \circ \s \colon M \to BU(1)$ under the isomorphism $H^2(M) \cong [M,BU(1)]$.  The map $\pr_1$ can be interpreted as a determinant, and $\mathcal{L}_\s$ is called the \emph{determinant line bundle of $\s$}.
The next lemma follows from~\cite[Proposition~2.4.16]{GompfStip}.

\begin{lemma}\label{spinc-lemma}
 Let $M$ be a closed, oriented 4-manifold.
\begin{enumerate}[(i)]
\item For every spin$^{c}$\hspace{-1pt} structure $\s$ on $M$, reduction modulo two is such that:
\begin{align*}
  H^2(M) &\to H^2(M;\Z/2) \\
  c_1(\s) & \mapsto w_2(M),
\end{align*}
where $w_2(M)$ is the second Stiefel-Whitney class.
\item There is a transitive action of $H^2(M)$ on the set of homotopy classes of spin$^{c}$\hspace{-1pt} structures on $M$, such that for $x \in H^2(M)$ we have
\[c_1(x \cdot \s) = c_1(\s) + 2x \in H^2(M).\]
\item If $H_1(M)$ is 2-torsion free, then this action is free.
\end{enumerate}
  \end{lemma}

  \begin{proof}
 As mentioned during the proof of Lemma~\ref{lemma:spin-c-exists}, in~\cite{GompfStip} spin$^{c}$ structures are defined by using $B\Spinc_4$ in place of $B\Spinc$ and therefore the Chern class of a spin$^{c}$-structure is defined using $\pr_1 \colon B\Spinc_4 \to BU(1)$.
 However, since the map $B\Spinc_4 \to BU(1)$ factors through~$B\Spinc$, both definitions of the Chern class coincide and so the lemma follows from \cite[Proposition~2.4.16]{GompfStip}.
  \end{proof}

As a consequence every characteristic cohomology class $y \in H^2(M)$ can be realised as the first Chern class of some spin$^{c}$\hspace{-1pt} structure on $M$, and if $H_1(M)$ is 2-torsion free then this spin$^{c}$\hspace{-1pt} structure is uniquely determined by~$y$.
Here recall that $y$ being \emph{characteristic} means that $\langle x \cup x,[M] \rangle \equiv \langle x\cup y ,[M] \rangle \mod{2}$ for every $x \in H^2(M)$.

The next lemma is immediate from the fact that the Chern class is an invariant of a spin$^{c}$\hspace{-1pt} structure, and is natural.

\begin{lemma}\label{lemma:c1-invariant}
If two spin$^{c}$\hspace{-1pt} structures $\s_1$ and $\s_2$ on a closed, oriented 4-manifold $M$ are equivalent, then there is an isometry of the intersection form
on $H^2(M)$ sending $c_1(\s_1)$ to $c_1(\s_2)$. \qed
\end{lemma}

To define stable equivalence of spin$^{c}$\hspace{-1pt} structures, fix once and for all the preferred spin$^{c}$\hspace{-1pt} structure $\s_g$ on
$W_g := \#^g S^2 \times S^2$, to be the spin$^{c}$\hspace{-1pt} structure with $c_1(\s_g) = 0 \in H^2(W_g)$. Such a spin$^{c}$\hspace{-1pt} structure exists by Lemma~\ref{spinc-lemma}.

\begin{definition}[Stable equivalence of spin$^c$ structures]\label{defn:stable-equiv-spinc}
Two spin$^{c}$\hspace{-1pt} structures $\s_1$ and $\s_2$ on a closed, oriented 4-manifold~$M$ are \emph{stably equivalent} if there exists $g \in \mathbb{N}_0$ such that the induced spin$^{c}$\hspace{-1pt} structures on $M \# W_g$, extending $\s_1$ and $\s_2$ using the fixed spin$^{c}$\hspace{-1pt} structure $\s_g$ on $W_g$, are equivalent.
\end{definition}

The stable classification of spin$^{c}$\hspace{-1pt} structures $\s$ on simply-connected
$4$-manifolds $M$ is analogous to the almost stable classification of $(4m{-}1)$-connected $8m$-manifolds from
Theorem~\ref{thm:stable-diffeo-classification-8-manifolds}.

We will want to apply Kreck's stable diffeomorphism theorem~\cite[Theorem~C]{KreckSurgeryAndDuality}, with appropriate $1$-smoothings.  In particular, a 1-smoothing has to be 2-connected.  While we will work with simply-connected 4-manifolds, so $\pi_1(M) =0 = \pi_1(B\Spinc)$, the map $M \to B\Spinc$ classifying a spin$^c$\hspace{-1pt} structure need not be surjective on $\pi_2$.  To mitigate this we make the following definition.

Given $(M, \s)$ we define the \emph{divisibility} $d(\s) \in \mathbb{N}_0$ of $c_1(\s)$ by the equation
\[ c_1(\s)(H_2(M)) = d(\s) \Z.\]
%
Let $B\Spinc(d)$ be the homotopy fibre of the mod~$d$ spin$^c$ first Chern class, so that there is a fibre sequence
\[B\Spinc(d) \xrightarrow{\pi} B\Spinc \to K(\Z/d, 2).\]
%
By construction, $\pi \colon B\Spinc(d) \to B\Spinc$ is a fibration, and the universal stable bundle over $B\Spinc$ pulls back to a stable bundle over $B\Spinc(d)$.

\begin{definition}
\label{def:Spinc-d}
Let $M$ be a closed, oriented $n$-manifold.
A \emph{spin$^{c}(d)$\hspace{-1pt} structure} on $M$ is a lift
\[\xymatrix{& B\Spinc(d) \ar[d]^{\pr_2 \circ \pi} \\ M \ar@{-->}[ur]^{\s(d)} \ar[r]^(0.45){\tau_M} & BSO}\]
of the stable tangent bundle's classifying map to $B\Spinc(d)$.
We denote a manifold with a $B\Spinc(d)$-structure by $(M, \s(d))$
and the corresponding bordism groups by $\Omega_*^{\Spinc(d)}$.
\end{definition}

\begin{lemma}
\label{lem:spinc(d)}
The following assertions hold:
\begin{enumerate}
\item $\pi_2(B\Spinc(d))=\Z$ for $d \neq 0$ and $\pi_2(B\Spinc(d))=0$ for $d=0$;
\item if $(M,\s)$ is a $\Spinc$-manifold, then $M$ is a $\Spinc(d)$-manifold for $d:=d(\mathfrak{s})$ and the map $\s(d) \colon M \to B\Spinc(d)$ is $2$-connected.
\end{enumerate}
\end{lemma}

\begin{proof}
We have $\pi_2(B\Spinc) \cong \Z$, and so the long exact sequence of a fibration in homotopy groups yields
\[0=\pi_3(K(\Z/d,2)) \to \Z \to \Z \twoheadrightarrow \Z/d = \pi_2(K(\Z/d,2)) \xrightarrow{} 0.\]
Since also $\pi_1(B\Spinc)=0$ we have that $B\Spinc(d)$ is $1$-connected and
$\pi_2(B\Spinc(d)) \cong \Z$ for $d\neq 0$.
A similar calculation shows that $\pi_2(B\Spinc(0))=\pi_2(B\Spin) =0.$ In fact it then follows from Whitehead's theorem that the map $B\Spin \to B\Spinc(0)$, obtained from factoring the canonical map $B\Spin \to B\Spinc$ through $B\Spinc(0)$, is a homotopy equivalence.
This concludes the proof of the first assertion.

We now assume that $(M,\s)$ is a $\Spinc$-manifold and prove the second assertion.
The first point follows by observing that in the exact sequence
$$ [M,B\Spinc(d)] \to [M,B\Spinc] \to [M,K(\Z/d,2)]=H^2(M;\Z/d),$$
the spin$^{c}$ structure $\mathfrak{s} \in  [M,B\Spinc]$ is mapped to zero, by definition of $d(\s)$.
It only remains to show that $\s(d)$ is $2$-connected.
Since this is clear for $d=0$, we assume that $d \neq 0$.
As we know that $\pi_2(B\Spinc)=\Z$ and $\pi_2(B\Spinc(d))=\Z$, the long exact sequence of the fibration and the definition of $d=d(\s)$ imply that $\im(\s_*)=d\Z \subseteq \Z=\pi_2(B\Spinc)$ and therefore $\s(d)$ is surjective on $\pi_2$, as required.
\end{proof}

Our aim is now to construct an injective map $\Omega_*^{\Spinc(d)} \to \Z \oplus \Z$.
The first component of this map will be the signature, while the second will arise as a characteristic number obtained from $\s(d) \colon M \to B\Spinc(d)$ by pulling back a universal class $c_1/d \in H^2( B\Spinc(d))$ that we now define.
For $d \neq 0$, Lemma~\ref{lem:spinc(d)} implies that $H^2(B\Spinc(d))$ is an infinite cyclic group.
It is generated by a class $c_1/d \in  H^2(B\Spinc(d))$
such that the pullback $\pi^*(c_1)$ of the spin$^c$ first Chern class, satisfies
\begin{equation}\label{equation:defining-prop-ofc1-over-d}\tag{$\varpi$}
  d \big( c_1/d \big) = \pi^*(c_1) \in H^2(B\Spinc(d)).
\end{equation}
For $d=0$, $H^2(B\Spinc(0)) = H^2(B\Spin) =0$, and we set $c_1/d =0$.
As is conventional for characteristic classes, given a $\Spinc(d)$-structure $\s(d) \colon M \to B\Spinc(d)$ we write $c_1/d(\s(d)) := \s(d)^*(c_1/d) \in H^2(M)$.

\begin{lemma} \label{lem:spincd}
There is an injective homomorphism
\begin{align*}
  \Theta \colon \Omega_4^{\Spinc(d)} &\to \Z \oplus \Z \\
[N,\s(d)]  & \mapsto  \big( \sigma(N), \langle(c_1/d(\s(d)))^2,[N]\rangle \big).
\end{align*}
\end{lemma}

%
%

\begin{proof}
The given map is a homomorphism, and is a bordism invariant because the signature is bordism invariant, and because $c_1^2$ is a characteristic number and therefore so is $(c_1/d)^2$.
%

It remains to prove injectivity of $\Theta$.
Let $(M, \s(d))$ be a spin$^c(d)$-manifold with vanishing signature and
$(c_1/d(\s(d)))^2 = 0$.
Since $\pi_1(B\Spinc(d)) = 0$,
after preliminary surgeries over $B\Spinc(d)$ 
we may assume that $M$ is simply-connected.
Since $\sigma(M) = 0$, the homeomorphism classification of smooth simply-connected
$4$-manifolds~\cite{Freedman-JDG-82} means that we can assume that $M$ is homeomorphic to one
of the following model manifolds:
\[ M \cong_{\rm TOP}
\begin{cases}
W_g & \text{$d$ even,} \\
X_g & \text{$d$ odd,}
\end{cases}\]
where $X_g := \#^g S^2 \wt{\times} S^2$.
In other words $M$ is a possibly exotic $W_g$ or $X_g$.
Now, exotic pairs of simply-connected 4-manifolds are $h$-cobordant~\cite[Theorem~2]{Wall-on-simply-conn-4mflds}, and the spin$^c(d)$-structure on $M$
propagates along an $h$-cobordism to a spin$^c(d)$ structure on either $W_g$ or $X_g$, as appropriate.  Hence we may assume
that $(M, \s(d))$ is diffeomorphic to either $(W_g, \s'_g(d))$ or $(X_g, \s''_g(d))$ for some spin$^c(d)$-structures
$\s'_g(d)$ or $\s''_g(d)$.
Now, $M$ has a standard coboundary $N$, $\partial N = M$, where
\[ N \cong
\begin{cases}
Y_g & \text{$d$ even,} \\
Z_g & \text{$d$ odd.}
\end{cases}\]
Here $Y_g := \natural^g D^3 \times S^2$ and $Z_g := \natural^g D^3 \widetilde{\times} S^2$, where $D^3 \widetilde{\times} S^2 \to S^2$ is the nontrivial bundle.
By assumption $(c_1/d(\s(d)))^2 = 0$ and it follows that $c_1/d(\s(d)) \in L$,
for some lagrangian $L \subseteq H^2(M)$.
Now, the automorphisms of the intersection form act transitively on the set of lagrangians (see for example \cite[pp.\ 144-5]{Wall-on-simply-conn-4mflds}), and Wall~\cite[p.\ 144]{Wall-on-simply-conn-4mflds} also showed that every isometry of the intersection form of $H^2(M)$ is realised by a diffeomorphism.
Hence we may assume that $c_1/d(\s(d)) \in H^2(M)$ lies in the standard lagrangian of $H^2(M)$, and so is the restriction to the boundary of $c$ for some $c \in H^2(N)$.
Since $H^2(N) \to H^2(M)$ is onto a summand, it follows that $N$ admits a spin$^c(d)$-structure $\s_{N}(d)$ that restricts to $\s(d)$.
Hence $(N, \s_{N}(d))$ is a spin$^c(d)$ null-bordism of $(M, \s)$, and so $\Theta$ is indeed injective.
\end{proof}

Next, using Lemma~\ref{lem:spincd} we deduce the stable classification of spin$^{c}$\hspace{-1pt} structures on simply-connected 4-manifolds.

  The fibration sequence defining $B\Spinc(d)$ gives rise to an exact sequence
  \[[M,\Omega K(\Z/d,2)] \to [M,B\Spinc(d)] \to [M,B\Spinc] \to [M,K(\Z/d,2)].\]
If $M$ is simply-connected then $[M,\Omega K(\Z/d,2)] \cong [M,K(\Z/d,1)] \cong H^1(M;\Z/d) =0$, so if a lift of $\Spinc$ structure $\s$ to a spin$^{c}(d)$ structure  $\s(d)$ exists, then it is essentially unique.

A spin$^{c}(d)$ structure $\s(d)$ on $M$ induces a spin$^{c}(d)$ structure on $M \# W_g$, for any $g$: as in Definition~\ref{defn:stable-equiv-spinc} we extend the associated $\Spinc(d)$ structure on $M$ by the spin$^c(d)$ structure on $W_g$ with $c_1=0$.  By the previous paragraph, since $W_g$ is simply connected, there is an essentially unique such $\Spinc(d)$ structure on $W_g$.
Then a lift to a spin$^{c}(d)$ structure on $M$ determines such a lift on $M \# W_g$.
We can therefore define stable equivalence of spin$^{c}(d)$ structures. The definition is identical to the definition for spin$^{c}(d)$ structure, just replacing $\Spinc$ with $\Spinc(d)$ throughout Definition~\ref{defn:equiv-of-spinc}~\eqref{item:defn-equiv-spinc} and Definition~\ref{defn:stable-equiv-spinc}.

\begin{theorem}\label{thm:stable-classification}
Let $M$ be a closed, oriented, simply-connected 4-manifold. Two spin$^{c}$\hspace{-1pt} structures $\s_1$ and $\s_2$ on $M$ are stably equivalent if and only if $d(\s_1) = d(\s_2) \in \mathbb{N}_0$ and  $c_1(\s_1)^2 = c_1(\s_2)^2 \in H^4(M)$.
\end{theorem}


\begin{proof}
For the forward direction, the square of the Chern class and its divisibility are preserved by stable equivalence because we fixed the spin$^{c}$\hspace{-1pt} structure on $W_g$ to be the structure with trivial first Chern class, and because equivalence of spin$^{c}$\hspace{-1pt} structures preserve Chern numbers and the divisibility.

For the reverse direction, by Lemma~\ref{lem:spincd}, for a fixed 4-manifold $M$, two spin$^{c}$\hspace{-1pt} structures $\s_1$ and $\s_2$ on $M$ with $d(\s_1) = d(\s_2) =d$ determine $B\Spinc(d)$-structures $\s_1(d)$ and $\s_2(d)$, and therefore elements of $\Omega_4^{\Spinc(d)}$.    Since  $c_1(\s_1)^2 = c_1(\s_2)^2$, it follows that $(c_1/d(\s_1(d)))^2 = (c_1/d(\s_2(d)))^2$: for $d=0$ this is automatic; for $d \neq 0$ apply $\pi^*$ to $c_1(\s_1)^2 = c_1(\s_2)^2$ and use~\eqref{equation:defining-prop-ofc1-over-d}. Therefore, since $\sigma(M)$ is independent of tangential structures, by Lemma~\ref{lem:spincd} $(M,\s_1(d))$ and $(M,\s_1(d))$ are bordant over $B\Spinc(d)$.
Then we apply Kreck's stable diffeomorphism theorem~\cite[Theorem~C]{KreckSurgeryAndDuality}, which in the current situation implies that $\Spinc(d)$ structures on $M$ are stably equivalent if they are bordant. Here we use that the maps $M \to B\Spinc(d)$ are 1-smoothings by Lemma~\ref{lem:spinc(d)}.  Via $\pi \colon B\Spinc(d) \to B\Spinc$, a stable equivalence of $\Spinc(d)$ structures determines a stable equivalence of $\Spinc$ structures.
\end{proof}

 %

 %

Now we are ready to prove the main result of this section.
Let $C \in \mathbb{Z}$ be such that $|C| \geq 16$ and $8 \mid C$. Define $P(C) := |\mathcal{P}_{C/8}|$, namely the number of distinct primes dividing $C/8$.
We consider $S^2 \times S^2$ with a fixed orientation. This determines an identification $H^4(S^2 \times S^2)=\Z$.

\begin{theorem}\label{thm:stable-cx-s2xs2}
 For every $C \in \Z$ with $|C| \geq 16$ and $8 \mid C$, there are $n:= 2^{P(C)-1}$ stably equivalent spin$^{c}$\hspace{-1pt} structures $\s_1,\dots,\s_n$ on $S^2 \times S^2$  with $c_1(\s_i)^2 = C \in H^4(S^2 \times S^2) = \Z$, that are all pairwise inequivalent.
\end{theorem}

\begin{proof}
Let $M := S^2 \times S^2$.
Let $x,y \in H^2(M) \cong \Z^2$ be generators dual to $[\pt \times S^2]$ and $[S^2 \times \pt]$ respectively.
So $xy =1 \in H^4(M)$ while $x^2=y^2=0$.  Henceforth we identify $H^4(M) \cong \Z$.
Let $Q:= C/8$. There are $P(C)$ prime powers dividing $Q$.
Up to switching the order and multiplying both by $-1$, there are $2^{P(C)-1}$ ways to write $Q$ as a product of coprime integers $Q =q_1q_2$.
For each such factorisation, let $\s_i$ be a spin$^{c}$\hspace{-1pt} structure with
\[c_1(\s_i) = 2q_1x + 2q_2 y.\]
Such spin$^{c}$\hspace{-1pt} structures exist by Lemma~\ref{spinc-lemma}: every characteristic element of $H^2(M)$ can be realised as the first Chern class of some spin$^{c}$\hspace{-1pt} structure.
Note that
\[c_1(\s_i)^2 =  8q_1q_2 = 8Q=C \in \Z = H^4(S^2 \times S^2)\]
and $d(\s_i)=2$ for every $i$.
Thus by Theorem~\ref{thm:stable-classification}, all the $\s_i$ are stably equivalent to one another.
But as we saw in the proof of Proposition~\ref{prop:ManifoldNab} there is no isometry of the intersection pairing of $M$ that sends $(2q_1,2q_2)$ to $(2q_1',2q_2')$ in $H^2(M) \cong \Z^2$ for distinct unordered pairs $\{q_1,q_2\}$ and $\{q_1',q_2'\}$.
By Lemma~\ref{lemma:c1-invariant} it follows that the $\{\s_i\}$ are pairwise inequivalent spin$^{c}$\hspace{-1pt} structures.
\end{proof}


\bibliographystyle{annotate}
\bibliography{biblio}

\begin{thebibliography}{{Wal}64}

\bibitem[Ada60]{Adams:1960}
J.~Frank Adams.
\newblock On the non-existence of elements of {H}opf invariant one.
\newblock {\em Ann. of Math. (2)}, 72:20--104, 1960.


\bibitem[Ada66]{Adams:1966-1}
J.~Frank Adams.
\newblock On the groups {$J(X)$}. {IV}.
\newblock {\em Topology}, 5:21--71, 1966.


\bibitem[BHS19]{Burklund-Hahn-Senger}
Robert Burklund, Jeremy Hahn, and Andrew Senger.
\newblock On the boundaries of highly connected, almost closed manifolds.
\newblock {\em Preprint, available at ar{X}iv:1910.14116}, 2019.


\bibitem[BM58]{Bott-Milnor}
Raoul Bott and John Milnor.
\newblock On the parallelizability of the spheres.
\newblock {\em Bull. Amer. Math. Soc.}, 64:87--89, 1958.


\bibitem[Bro72]{Bro72}
William Browder.
\newblock {\em Surgery on simply-connected manifolds}.
\newblock Springer-Verlag, New York, 1972.
\newblock Ergebnisse der Mathematik und ihrer Grenzgebiete, Band 65.


\bibitem[Bru68]{Brumfiel-I}
Gregory Brumfiel.
\newblock {On the homotopy groups of BPL and PL/O}.
\newblock {\em {Ann. Math. (2)}}, 88:291--311, 1968.


\bibitem[BS20]{Burklund-Senger}
Robert Burklund and Andrew Senger.
\newblock On the high-dimensional geography problem.
\newblock {\em Preprint, available at ar{X}iv:2007.05127}, 2020.


\bibitem[CCPS21]{CCPS2}
Anthony Conway, Diarmuid Crowley, Mark Powell, and Joerg Sixt.
\newblock Stably diffeomorphic manifolds and modified surgery obstructions.
\newblock {\em Preprint; available at ar{X}iv:2109.05632}, 2021.


\bibitem[CS11]{CrowleySixt}
Diarmuid Crowley and Joerg Sixt.
\newblock Stably diffeomorphic manifolds and {$l_{2q+1}(\Bbb Z[\pi])$}.
\newblock {\em Forum Math.}, 23(3):483--538, 2011.


\bibitem[Dav05]{DavisBorelNovikov}
James Davis.
\newblock The {B}orel/{N}ovikov conjectures and stable diffeomorphisms of
  4-manifolds.
\newblock In {\em Geometry and topology of manifolds}, volume~47 of {\em Fields
  Inst. Commun.}, pages 63--76. Amer. Math. Soc., Providence, RI, 2005.

\bibitem[Don83]{Donaldson}
Simon~K. Donaldson.
\newblock An application of gauge theory to four-dimensional topology.
\newblock {\em J. Differential Geom.}, 18(2):279--315, 1983.


\bibitem[Fre82]{Freedman-JDG-82}
Michael Freedman.
\newblock The topology of four-dimensional manifolds.
\newblock {\em J. Differential Geom.}, 17(3):357--453, 1982.


\bibitem[GS99]{GompfStip}
Robert Gompf and Andr\'as Stipsicz.
\newblock {\em {$4$}-manifolds and {K}irby calculus}, volume~20 of {\em
  Graduate Studies in Mathematics}.
\newblock American Mathematical Society, Providence, RI, 1999.


\bibitem[Hae61]{Haefliger}
Andr\'{e} Haefliger.
\newblock Differentiable imbeddings.
\newblock {\em Bull. Amer. Math. Soc.}, 67:109--112, 1961.


\bibitem[Hae62]{HaefligerDifferentiableLinks}
Andr\'{e} Haefliger.
\newblock Differentiable links.
\newblock {\em Topology}, 1:241--244, 1962.


\bibitem[HK88a]{HambletonKreckFiniteGroup}
Ian Hambleton and Matthias Kreck.
\newblock On the classification of topological {$4$}-manifolds with finite
  fundamental group.
\newblock {\em Math. Ann.}, 280(1):85--104, 1988.


\bibitem[HK88b]{HambletonKreckSmoothStructures}
Ian Hambleton and Matthias Kreck.
\newblock Smooth structures on algebraic surfaces with cyclic fundamental
  group.
\newblock {\em Invent. Math.}, 91(1):53--59, 1988.


\bibitem[HK93a]{HambletonKreckCancellation3}
Ian Hambleton and Matthias Kreck.
\newblock Cancellation, elliptic surfaces and the topology of certain
  four-manifolds.
\newblock {\em J. Reine Angew. Math.}, 444:79--100, 1993.


\bibitem[HK93b]{HambletonKreckCancellation2}
Ian Hambleton and Matthias Kreck.
\newblock Cancellation of hyperbolic forms and topological four-manifolds.
\newblock {\em J. Reine Angew. Math.}, 443:21--47, 1993.


\bibitem[IR90]{Ireland-Rosen}
Kenneth Ireland and Michael Rosen.
\newblock {\em A classical introduction to modern number theory}, volume~84 of
  {\em Graduate Texts in Mathematics}.
\newblock Springer-Verlag, New York, second edition, 1990.


\bibitem[JW54]{James-Whitehead-I}
I.~M. James and J.~H.~C. Whitehead.
\newblock The homotopy theory of sphere bundles over spheres. {I}.
\newblock {\em Proc. London Math. Soc. (3)}, 4:196--218, 1954.


\bibitem[KM63]{KeMi63}
Michel~A. Kervaire and John~W. Milnor.
\newblock Groups of homotopy spheres. {I}.
\newblock {\em Ann. of Math. (2)}, 77:504--537, 1963.


\bibitem[KR20]{Krannich-Reinhold}
Manuel Krannich and Jens Reinhold.
\newblock Characteristic numbers of manifold bundles over surfaces with highly
  connected fibers.
\newblock {\em Journal of the London Mathematical Society}, 102(2):879--904,
  2020.


\bibitem[Kre99]{KreckSurgeryAndDuality}
Matthias Kreck.
\newblock Surgery and duality.
\newblock {\em Ann. of Math. (2)}, 149(3):707--754, 1999.


\bibitem[KS84]{KreckSchafer}
Matthias Kreck and James~A. Schafer.
\newblock {Classification and stable classification of manifolds: some
  examples.}
\newblock {\em {Comment. Math. Helv.}}, 59:12--38, 1984.


\bibitem[Lev85]{LevineLectures}
J.~P. Levine.
\newblock Lectures on groups of homotopy spheres.
\newblock In {\em Algebraic and geometric topology ({N}ew {B}runswick,
  {N}.{J}., 1983)}, volume 1126 of {\em Lecture Notes in Math.}, pages 62--95.
  Springer, Berlin, 1985.

\bibitem[L{\"{u}}c02]{lueck-surgery-intro}
Wolfgang L{\"{u}}ck.
\newblock A basic introduction to surgery theory.
\newblock In {\em Topology of high-dimensional manifolds, {N}o. 1, 2
  ({T}rieste, 2001)}, volume~9 of {\em ICTP Lect. Notes}, pages 1--224. Abdus
  Salam Int. Cent. Theoret. Phys., Trieste, 2002.

\bibitem[Mil58a]{Milnor-simply-connected-4-manifolds}
John Milnor.
\newblock On simply connected {$4$}-manifolds.
\newblock In {\em Symposium internacional de topolog\'\i a algebraica
  {I}nternational symposi um on algebraic topology}, pages 122--128.
  Universidad Nacional Aut\'onoma de M\'exico and UNESCO, Mexico City, 1958.

\bibitem[Mil58b]{Milnor1958}
John Milnor.
\newblock On the {W}hitehead homomorphism {$J$}.
\newblock {\em Bull. Amer. Math. Soc.}, 64:79--82, 1958.


\bibitem[MK60]{Kervaire-Milnor-ICM}
John~W. Milnor and M.~Kervaire.
\newblock Bernoulli numbers, homotopy groups, and a theorem of {R}ohlin.
\newblock In {\em Proc. {I}nternat. {C}ongress {M}ath. 1958}, pages 454--458.
  Cambridge Univ. Press, New York, 1960.

\bibitem[Qui71]{Quillen71}
Daniel Quillen.
\newblock The {A}dams conjecture.
\newblock {\em Topology}, 10:67--80, 1971.


\bibitem[Sco05]{Scorpan}
Alexandru Scorpan.
\newblock {\em The wild world of 4-manifolds}.
\newblock American Mathematical Society, Providence, RI, 2005.


\bibitem[{Sto}85]{Stolz-highly-conn-mflds-and-their-boundaries}
Stephan {Stolz}.
\newblock {\em {Hochzusammenh\"angende {M}annigfaltigkeiten und ihre
  {R}\"ander}}, volume 1116.
\newblock Springer, Cham, 1985.


\bibitem[Tod62]{Toda}
Hirosi Toda.
\newblock {\em Composition methods in homotopy groups of spheres}.
\newblock Annals of Mathematics Studies, No. 49. Princeton University Press,
  Princeton, N.J., 1962.


\bibitem[Wal62]{Wall-classification-n-1-conn-2n}
C.~T.~C. Wall.
\newblock Classification of {$(n{-}1)$}-connected {$2n$}-manifolds.
\newblock {\em Ann. of Math. (2)}, 75:163--189, 1962.


\bibitem[{Wal}64]{Wall-on-simply-conn-4mflds}
C.~T.~C. {Wall}.
\newblock {On simply-connected 4-manifolds}.
\newblock {\em {J. Lond. Math. Soc.}}, 39:141--149, 1964.


\bibitem[Wal67]{Wall-n-1-conn-2n+1}
C.~T.~C. Wall.
\newblock Classification problems in differential topology. {VI}.
  {C}lassification of {$(s{-}1)$}-connected {$(2s{+}1)$}-manifolds.
\newblock {\em Topology}, 6:273--296, 1967.


\bibitem[Wal99]{WallSurgery}
Charles Terence~Clegg Wall.
\newblock {\em Surgery on compact manifolds}.
\newblock American Mathematical Society, Providence, RI, second edition, 1999.
\newblock Edited and with a foreword by A. A. Ranicki.


\bibitem[Whi42]{Whitehead-defn-J}
George~W. Whitehead.
\newblock On the homotopy groups of spheres and rotation groups.
\newblock {\em Ann. of Math. (2)}, 43:634--640, 1942.


\bibitem[Whi49]{Whitehead-4-complexes}
J.~H.~C. Whitehead.
\newblock On simply connected, {$4$}-dimensional polyhedra.
\newblock {\em Comment. Math. Helv.}, 22:48--92, 1949.


\end{thebibliography}

\end{document}